\documentclass{amsart}

\usepackage{amscd,amsmath,amssymb,amsfonts,bbm, calligra, xspace}
\usepackage[T1]{fontenc}
\usepackage{lmodern}
\usepackage{mathtools}
\usepackage{enumerate}
\usepackage{multicol}
\usepackage[all]{xy}
\usepackage{mathrsfs}
\usepackage{footmisc}
\usepackage[unicode,pdfborder={0 0 0},final]{hyperref}

\newtheorem{thm}{Theorem}[section]
\newtheorem{prop}[thm]{Proposition}

\newtheorem{lem}[thm]{Lemma}
\newtheorem{cor}[thm]{Corollary}

\theoremstyle{definition}

\theoremstyle{remark}
\newtheorem{rem}[thm]{Remark}
\newtheorem{rems}[thm]{Remarks}

\newtheorem{Step}{Step}

\numberwithin{equation}{section}

\newcommand{\Id}{\mathrm{Id}}
\newcommand{\Rr}{\mathrm{R}}
\newcommand{\Cl}{\mathrm{Cl}}

\newcommand{\Grass}{\mathrm{Grass}}
\newcommand{\Bl}{\mathrm{Bl}}
\newcommand{\MO}{\mathrm{MO}}
\newcommand{\BO}{\mathrm{BO}}

\newcommand{\Ker}{\mathrm{Ker}}

\newcommand{\cl}{\mathrm{cl}}

\newcommand{\van}{\mathrm{van}}

\newcommand{\et}{\mathrm{\acute{e}t}}
\newcommand{\cont}{\mathrm{cont}}
\newcommand{\Et}{\mathrm{\acute{E}t}}

\newcommand{\nr}{\mathrm{nr}}
\newcommand{\pt}{\mathrm{pt}}

\newcommand{\KK}{\mathrm{K}}

\newcommand{\GL}{\mathrm{GL}}
\newcommand{\SO}{\mathrm{SO}}
\newcommand{\SL}{\mathrm{SL}}

\newcommand{\Hom}{\mathrm{Hom}}
\newcommand{\Spec}{\mathrm{Spec}}

\newcommand{\Hdg}{\mathrm{Hdg}}
\newcommand{\Gal}{\mathrm{Gal}}

\newcommand{\MU}{\mathrm{MU}}
\newcommand{\Sq}{\mathrm{Sq}}

\def\alg{\mathrm{alg}}

\newcommand{\RR}{\mathrm{R}}

\newcommand{\isoto}{\myxrightarrow{\,\sim\,}}
\makeatletter
\def\myrightarrow{{\setbox\z@\hbox{$\rightarrow$}\dimen0\ht\z@\multiply\dimen0 6\divide\dimen0 10\ht\z@\dimen0\box\z@}}
\def\myrightarrowfill@{\arrowfill@\relbar\relbar\myrightarrow}
\newcommand{\myxrightarrow}[2][]{\ext@arrow 0359\myrightarrowfill@{#1}{#2}}
\makeatother

\def\loccit{\emph{loc}.\kern3pt \emph{cit}.{}\xspace}
\def\eg{e.g.\kern.3em}

\def\resp {\text{resp.}\kern.3em}

\def\bG{\mathbf G}

\def\bmu{\boldsymbol\mu}

\newcommand{\CH}{\mathrm{CH}}

\def\cA{\mathcal{A}}

\def\cL{\mathcal{L}}
\def\cO{\mathcal{O}}

\def\cH{\mathcal{H}}

\def\cU{\mathcal{U}}
\def\cV{\mathcal{V}}
\def\cW{\mathcal{W}}
\def\cY{\mathcal{Y}}

\def\cZ{\mathcal{Z}}

\def\ci{\mathcal{C}^{\infty}}

\def\kX{\mathfrak{X}}

\def\talpha{\tilde{\alpha}}

\def\oF{\overline{\mathbb{F}}_p}

\def\oZ{\overline{Z}}

\def\oX{\overline{X}}

\def\oW{\overline{W}}

\def\oeta{\overline{\eta}}

\def\wW{\widetilde{W}}

\def\wY{\widetilde{Y}}
\def\wcY{\widetilde{\mathcal{Y}}}

\def\wH{\widetilde{H}}

\def\A{\mathbb A}

\def\Z{\mathbb Z}
\def\C{\mathbb C}
\def\F{\mathbb F}

\def\Q{\mathbb Q}
\def\P{\mathbb P}
\def\R{\mathbb R}

\def\bN{\mathbb N}

\begin{document}

\title[Steenrod operations and algebraic classes]{Steenrod operations and algebraic classes}

\author{Olivier Benoist}
\address{D\'epartement de math\'ematiques et applications, \'Ecole normale sup\'erieure, CNRS,
45 rue d'Ulm, 75230 Paris Cedex 05, France}
\email{olivier.benoist@ens.fr}

\renewcommand{\abstractname}{Abstract}
\begin{abstract}
Based on a relative Wu theorem in \'etale cohomology, we study the compatibility of Steenrod operations on Chow groups and on \'etale cohomology. Using the resulting obstructions to algebraicity, we construct new
examples of non-algebraic cohomology classes over various fields ($\C$, $\R$, $\oF$, $\F_q$).

We also use Steenrod operations to study the mod $2$ cohomology classes of a compact $\ci$ manifold $M$ that are algebraizable, \emph{i.e.} algebraic on some real algebraic model of $M$. 
We give new examples of algebraizable and non-algebraizable classes, answering questions of Benedetti,  Ded\`o and Kucharz.
\end{abstract}
\maketitle

\section{Introduction}

The theme of this article is the interaction between Steenrod operations and algebraic cohomology classes of smooth projective varieties.

\subsection{Action of Steenrod operations on algebraic classes}

Our starting point is a theorem of Kawai \cite{Kawai} according to which mod $\ell$ Steenrod operations preserve algebraic cohomology classes with $\Z/\ell$ coefficients on smooth projective complex varieties (see also \cite[Example 19.1.8]{Fulton}). This implies in particular that mod $\ell$ algebraic classes on smooth projective complex varieties are killed by odd degree Steenrod operations, and hence provides topological obstructions to the algebraicity of such classes. One recovers in this way the first counterexamples to the integral Hodge conjecture, constructed by Atiyah and Hirzebruch \cite{AH}.

These arguments were adapted by Colliot-Th\'el\`ene and Szamuely with help from Totaro \cite[Theorem~2.1]{CTS} to construct counterexamples to the integral Tate conjecture over $\oF$ (whose statement is recalled in \S\ref{conventions}). We refer to \cite{Quicktorsion, PY, Kameko, Antieau} for further examples. In particular, the vanishing of odd degree Steenrod operations on mod~$\ell$ algebraic classes over algebraically closed fields of characteristic $\neq \ell$, proven in a particular case in \cite[Theorem~2.1~(1)]{CTS},
follows in general from Quick's work \cite[Theorem 1.1]{Quicktorsion}.
Our first goal is to extend this result to arbitrary, not necessarily algebraically closed, fields.

\vspace{.5em}

To do so, we compare Brosnan's Steenrod operations on Chow groups (\cite{Brosnan}, see also \S\ref{parSteenrodChow}) and the Steenrod operations in (twisted) \'etale cohomology (see \cite{GW} and \S\S\ref{Steenrodetale}--\ref{Steenrodtwistedetale}) by means of the mod $\ell$ \'etale cycle class map $\cl$.

In the next theorem, we write $\Sq$ for the total Steenrod square (when ${\ell=2}$) and~$P$ for the total Steenrod $\ell$\nobreakdash-th power (when $\ell$ is odd). 
When $\ell=2$, the cohomology class $\varpi$ is the class of~$-1$ in $k^*/(k^*)^2\simeq H^1_{\et}(\Spec(k),\Z/2)$.  See also Remark~\ref{rems}~(ii).

\begin{thm}[Theorem \ref{compatibility}]
\label{compatibilityintro}
Let $X$ be a smooth quasi-projective variety over a field $k$. Let $\ell$ be a prime number invertible in $k$. For $x\in \CH^c(X)/\ell$, one~has
\begin{alignat*}{5}
\Sq(\cl(x))&=&\,&\sum_{i\geq 0}(1+\varpi)^{c-i}\cdot\cl(\Sq^{2i}(x))&\,&\textrm{ \hspace{.5em}if }\ell=2,\\
\nonumber P(\cl(x))&=&\,&\cl(P(x))&\,&\textrm{ \hspace{.5em}if }\ell\textrm{ is odd.}\
\end{alignat*}
\end{thm}

It follows from Theorem~\ref{compatibilityintro} that algebraic cohomology classes are preserved by all Steenrod operations as soon as $k$ contains a primitive $\ell^2$-th root of unity, and by Steenrod's $\ell$-th powers when $\ell$ is odd and $k$ is arbitrary (see Proposition \ref{stability}). It is however not true that Steenrod squares always preserve algebraic classes (\emph{e.g.} if $\ell=2$ and $k=\R$, see Remark \ref{remSteencl}), and one can view Theorem \ref{compatibilityintro} as providing the correct substitute for this incorrect statement.

Theorem \ref{compatibilityintro} is proved by comparing two relative Wu theorems (describing the behaviour of Steenrod operations under proper pushforwards): one for Chow groups and another one in \'etale cohomology. The former is due to Brosnan (\cite{Brosnan}, see Proposition~\ref{WuChow}). The latter is due to Scavia and Suzuki \cite[Proposition 3.1]{ScaSuz2} when $\ell=2$ and $k$ is algebraically closed, and could be deduced in general from the Riemann--Roch theorems of Panin \cite{Panin} and D\'eglise \cite{Deglise}. We have found it more convenient to give a self-contained proof in \S\ref{parWu} (see Theorem \ref{thWu}).


\subsection{Examples of non-algebraic cohomology classes}

Our second goal, following the above-mentioned counterexamples to the integral Hodge and Tate conjectures, is to exploit the obstructions to algebraicity stemming from Theorem \ref{compatibilityintro} to give interesting new examples of non-algebraic cohomology classes.

\subsubsection{} 

We first apply classical topological obstructions to algebraicity going back to Atiyah and Hirzebruch \cite{AH}  (the vanishing of odd degree Steenrod operations) to a variety constructed in \cite[\S 5.1]{BOConiveau}. This yields a $5$-dimensional counterexample to the integral Tate conjecture over $\oF$ (see Remark \ref{rem1} for its Hodge counterpart).


\begin{thm}[Theorem \ref{th1b}]
\label{th1bintro}
Let $p$ be an odd prime number. There exists a smooth projective variety of dimension $5$ over $\oF$ on which the $2$-adic integral Tate conjecture for codimension $2$ cycles fails.
\end{thm}

The counterexamples to the integral Tate conjecture over $\oF$ that have appeared so far in the literature all have dimension $\geq 7$. 
By \cite[(5.10)]{Tate} and \cite[Theorem 0.5]{Schoen}, it would follow from the (rational) Tate conjecture over $\oF$ that no $3$-dimensional counterexample exists. We do not know any $4$-dimensional counterexample.
In fact, we do not know any counterexample to the integral Tate conjecture for cycles of dimension $2$ over $\oF$.

\subsubsection{} 

We then notice that the stability of algebraic classes by Steenrod operations (over, say, algebraically closed fields) yields obstructions to algebraicity that go beyond the vanishing of odd degree Steenrod operations. 
For instance, over the complex numbers, it implies that even degree Steenrod operations send mod $\ell$ algebraic classes to reductions mod $\ell$ of integral Hodge classes (and, over $\oF$, to reductions mod $\ell$ of integral Tate classes). This really may obstruct the integral Hodge and Tate conjectures, because there are no reasons why reductions mod $\ell$ of integral Hodge classes (or of integral Tate classes) should be preserved by even degree Steenrod operations. We illustrate these new Hodge-theoretic (or Galois-theoretic) obstructions to the integral Hodge and Tate conjectures by using them to prove the following theorem (see Remark \ref{rem2} for its Hodge counterpart).

\begin{thm}[Theorem \ref{th2b}]
\label{th2bintro}
Choose $\ell\in\{2,3,5\}$ and let $p$ be a prime number distinct from~$\ell$.
There exists a smooth projective variety $X$ of dimension $2\ell+2$ over $\oF$ with $H^*_{\et}(X,\Z_{\ell})$ torsion-free on which the $\ell$-adic integral Tate conjecture for codimension~$2$ cycles fails.
\end{thm}

Theorem \ref{th2bintro} gives the first counterexample to the integral Tate conjecture over~$\oF$ with a torsion-free cohomology ring. Indeed, the topological obstructions used in the counterexamples that have previously appeared in the literature force their $\ell$-adic cohomology rings to contain non-zero torsion classes (although the non-algebraic integral Tate classes of \cite{PY} live in torsion-free $\ell$-adic cohomology groups). 
We also refer to \cite[Theorem 1.10]{dJP} for more recent examples, valid for all prime numbers $\ell$, and based on entirely different arguments.

Our proof builds upon \cite{PY}. 
In this article, Pirutka and Yagita use algebraic approximations of classifying spaces of exceptional reductive groups to construct interesting $(2\ell+3)$\nobreakdash-dimensional counterexamples to the $\ell$-adic integral Tate conjecture over~$\oF$ by the Atiyah--Hirzebruch method, for $\ell\in\{2,3,5\}$. We notice that a well-chosen hypersurface  in (a blow-up of) their variety has a torsion-free cohomology ring, but the integral Tate conjecture may still fail by the obstruction described above. 
 
In \cite{Totarobordism}, Totaro explained how to better understand and generalize the Atiyah--Hirzebruch topological obstructions to the integral Hodge conjecture using complex cobordism. In \S\ref{parMU}, we show that our Hodge-theoretic enhancement of these obstructions also has such an interpretation, in terms of the Hodge classes in complex cobordism introduced by Hopkins and Quick \cite{HQ}.

\subsubsection{} 

Finally, we give examples of non-algebraic cohomology classes on varieties defined over non-closed fields, that become algebraic (and even vanish) over the algebraic closure of the base field. This is particularly interesting because their non-algebraicity is then a purely arithmetic phenomenon. 

Over finite fields, the first examples of such classes were constructed very recently by Scavia and Suzuki \cite[Theorem 1.3]{ScaSuz}. As they relied on obstructions to algebraicity induced by non-vanishing odd degree Steenrod operations, they had to restrict to finite fields containing a primitive $\ell^2$-th root of unity. Based on the subtler obstructions to algebraicity stemming from Theorem~\ref{compatibilityintro}  (see Proposition~\ref{weirdobstr}), we are able to remove this hypothesis.

\begin{thm}[Theorem \ref{finitethm}]
\label{finitethmintro}
Let $p\neq \ell$ be prime numbers, and let $\F$ be a finite subfield of $\oF$.
There exist a smooth projective geometrically connected
variety $X$ of dimension $2\ell+3$ over~$\F$ and a non-algebraic class 
$$x\in \Ker\big(H^4_{\et}(X,\Z_{\ell}(2))\to H^4_{\et}(X_{\oF},\Z_{\ell}(2))\big).$$
\end{thm}

Scavia and Suzuki's examples are algebraic approximations of classifying spaces of semisimple algebraic groups. To be able to encompass all finite fields, we instead resort to algebraic approximations of classifying spaces of (not necessarily constant) finite \'etale group schemes.

Over the field $\R$ of real numbers, non-algebraic classes that are geometrically algebraic are harder to construct on real varieties with no real points (as there are then no obstructions to algebraicity induced by the topology of the set of real points). The example of lowest dimension appearing in the literature is the $7$-dimensional anisotropic quadric (\cite[Example~2.5]{BW1}, see also \cite[\S 5]{Yagitaquadrics}). Another application of Proposition \ref{weirdobstr} allows us to produce a $4$-dimensional example. In the next statement, we set $G:=\Gal(\C/\R)$ and we let $H^*_G(X(\C),-)$ denote the $G$-equivariant cohomology of the set of complex points of a smooth variety over $\R$.

\begin{thm}[Theorem \ref{threal}]
\label{threalintro}
There exists a smooth projective variety $X$ of dimension $4$ over~$\R$ such that $X(\R)=\varnothing$,
and a non-algebraic class 
$$x\in \Ker\big(H^4_{G}(X(\C),\Z(2))\to H^4(X(\C),\Z(2))\big).$$
\end{thm}

  Theorem \ref{threalintro} gives a new counterexample to the real integral Hodge conjecture introduced in \cite[Definition 2.2]{BW1}, which is not explained by a failure of the usual (complex) integral Hodge conjecture
(see Corollary \ref{corIHCreal}). We refer to Remarks \ref{remIHCR} for a comparison with previously known counterexamples to the real integral Hodge conjecture.

\subsection{Algebraizability of cohomology classes of \texorpdfstring{$\ci$}{smooth} manifolds}
\label{realintro}

In the last section of this article, we consider questions specific to real algebraic geometry, that are thematically related but logically independent of the above.

Let $X$ be a smooth variety over $\R$. Borel and Haefliger (\cite{BH}, see also \cite[\S 1.6.2]{BW1}) have constructed a cycle class map 
\begin{equation*}
\label{BHccm}
\cl_{\R}:\CH^c(X)\to H^c(X(\R),\Z/2),
\end{equation*}
associating with a codimension $c$ integral subvariety $Z\subset X$ the fundamental class of its real locus. Computing or controlling the image $H^*_{\alg}(X(\R),\Z/2)$ of $\cl_{\R}$ is a central topic in real algebraic geometry (see \emph{e.g.} \cite[\S 11.3]{BCR} or \cite{BKHomology}). 

Let $M$ be a compact $\ci$ manifold of dimension $d$. The Nash--Tognoli theorem \cite[Corollario p.\,176]{Tognoli} asserts that there exists a smooth projective variety~$X$ of dimension $d$ over $\R$ and a diffeomorphism $\chi: M\isoto X(\R)$. We call such a pair~$(X,\chi)$ an \textit{algebraic model} of $M$.  A cohomology class $x\in H^*(M,\Z/2)$ is said to be \textit{algebraizable} if it belongs to $\chi^*H^*_{\alg}(X(\R),\Z/2)$ for some algebraic model $(X,\chi)$ of $M$. 
The problem of determining which cohomology classes are algebraizable, a cohomological strengthening of the Nash--Tognoli theorem, is still largely open.

 It is known since Benedetti and Ded\`o \cite[Theorem 1]{BD} that there exist non-algebraizable cohomology classes. Teichner \cite[Theorem~1]{Teichner} has found examples in dimension $6$ (the smallest dimension possible),
 and Kucharz \cite[Theorem~1.13]{Kuchomo} has discovered examples of all even degrees $\geq 2$. It remained an open problem, raised by Benedetti and Ded\`o \cite[p.~150]{BD} and Kucharz \cite[p.\,194]{KucNash}, to determine the possible degrees of non-algebraizable classes, and in particular to construct non-algebraizable classes of odd degree. We solve this problem entirely.

\begin{thm}[Theorem \ref{propnonalgebraizable}]
\label{propnonalgebraizableintro}
For all $c\geq 2$, there exists a compact $\ci$ manifold~$M$ and a class $x\in H^c(M,\Z/2)$ that is not algebraizable.
\end{thm}

Examples of algebraizable cohomology classes include Stiefel--Whitney classes of topological real vector bundles on $M$ and Poincar\'e-duals of fundamental classes of $\ci$ submanifolds of $M$ (see \cite[Theorem 4.2 and Corollary 4.5]{BT} 
or \cite[Theorem~2.10]{AKrelNash}).
Following Kucharz \cite{KucNash}, let us denote by $A(M)\subset H^*(M,\Z/2)$ the subring generated by these classes. It was asked by Benedetti and Ded\`o \cite[Conjecture 2 p.\,150]{BD} 
and again by Kucharz in \cite[Conjecture A]{KucNash} whether all algebraizable classes belong to~$A(M)$. We give a negative answer to this conjecture.

\begin{thm}[Theorem \ref{cexSWsub}]
\label{cexSWsubintro}
There exists a compact $\ci$ manifold $M$ and an algebraizable class $x\in H^5(M,\Z/2)$ that does not belong to $A(M)$.
\end{thm}

The proofs of Theorems \ref{propnonalgebraizableintro} and \ref{cexSWsubintro} rely on the stability of algebraizable classes by mod $2$ Steenrod operations, a result due to Akbulut and King (\cite[Theorem~6.6]{AK}, see also Corollary \ref{BHSteencor}).
To prove Theorem \ref{propnonalgebraizableintro}, we construct cohomology classes whose images by a Steenrod operation can be shown to be non-algebraizable  by the method of Kucharz \cite{Kuchomo}, hence are themselves not algebraizable. 
As for the proof of Theorem \ref{cexSWsubintro}, it relies on the observation that the subring $A(M)$ of $H^*(M,\Z/2)$ is in general not stable under the action of Steenrod operations.

\subsection{Conventions} 
\label{conventions}
A variety over a field $k$ is a separated $k$-scheme of finite type.

Let $X$ be a smooth projective variety over the algebraic closure $k$ of a finitely generated field $k_0$. A class in $H^{2c}_{\et}(X,\Z_{\ell}(c))$ is said to be a Tate class if it belongs to the subgroup 
$H^{2c}_{\et}(X_1\times_{k_1}k,\Z_{\ell}(c))^{G_{k_1}}$ of $H^{2c}_{\et}(X,\Z_{\ell}(c))$
associated with some model~$X_1$ of $X$ over a finite extension $k_1$ of~$k_0$ with absolute Galois group $G_{k_1}$. The classes that lie in the image of the cycle class map $\CH^c(X)\to H^{2c}_{\et}(X,\Z_{\ell}(c))$ are said to be algebraic, and are automatically Tate classes. In this article, the converse statement that all Tate classes are algebraic is called the integral Tate conjecture over $k$ (see \cite[\S 1]{CTScavia} for a discussion of variants of this statement).

\subsection{Acknowledgements} 
I thank Riccardo Benedetti, Jean-Louis Colliot-Th\'e\-l\`ene, Fr\'ed\'eric D\'eglise, Burt Totaro, Olivier Wittenberg and Nobuaki Yagita for helpful comments on a previous version of this article.

\section{A relative Wu theorem in \'etale cohomology}
\label{secWu}

In topology, the behaviour of Steenrod operations with respect to proper pushforwards is controlled by a Grothendieck--Riemann--Roch-type theorem: the relative Wu theorem of Atiyah and Hirzebruch \cite[Satz 3.2]{AHOp}.
In this section, we prove an analogue of this theorem in the \'etale cohomology of schemes. 

We first recall, in \S\ref{Steenrodetale} and \S\ref{Steenrodtwistedetale}, a construction of Steenrod operations on the (twisted) \'etale cohomology of schemes.  The characteristic classes appearing in the statement of the relative Wu theorem are defined and studied in \S\ref{characteristicclasses}, and the relative Wu theorem is proved in \S\ref{parWu}.

In this whole section we fix a prime number $\ell$ and we set $S_{\ell}:=\Spec(\Z[\frac{1}{\ell}])$.

\subsection{Steenrod operations in \'etale cohomology}
\label{Steenrodetale}

Let $\cA_{\ell}$ be the mod $\ell$ Steenrod algebra \cite[I \S 3, VI \S 2]{ES}.  It is a graded $\Z/{\ell}$\nobreakdash-algebra generated by the Steenrod squares $(\Sq^i)_{i\geq 1}$ which have degree $i$ (when $\ell=2$),  or by Steenrod's $\ell$-th powers $(P^i)_{i\geq 1}$ which have degree $2i(\ell-1)$ and by the Bockstein $\beta$ which has degree $1$ (when $\ell$ is odd), subject to the Adem relations. 
We write $\Sq^0=P^0=1$.  If $\ell=2$,  we set $\beta:=\Sq^1$.
We use the notation $\Sq:=\sum_{i\geq 0} \Sq^i$ and $P:=\sum_{i\geq 0}P^i$.

  The algebra $\mathcal{A}_{\ell}$ acts functorially on the mod $\ell$ (relative) cohomology of (pairs of) topological spaces \cite[VII, VIII]{ES}.
There are several ways to construct analogous Steenrod operations on the mod $\ell$ \'etale cohomology of schemes \cite{Epstein, May,Jardine, Feng}.  We follow Feng \cite[\S 3]{Feng} and Scavia and Suzuki \cite[\S 2]{ScaSuz} in exploiting \'etale homotopy theory \cite{Friedlander}.  This choice will allow us to use Scavia and Suzuki's analysis of the behaviour of Steenrod operations in Hochschild--Serre spectral sequences in \S\ref{cexfinite}.

Friedlander associates with any Noetherian scheme $X$ a pro-simplicial set denoted by $\Et(X)$ in \cite{Feng}: its \'etale topological type \cite[Definition~4.4]{Friedlander}. 
Let $f:Y\hookrightarrow X$ be a closed immersion of Noetherian schemes with complement open immersion $j:U\to X$.
Let $M(\Et(j))$ be the mapping cylinder in the sense of \cite[p.~140]{Friedlander} of the map $\Et(j):\Et(U)\to\Et(X)$ induced by~$j$. It is a pro-simplicial set endowed with an inclusion  $\Et(U)\hookrightarrow M(\Et(j))$.
For any abelian group $A$, it follows from \cite[Corollary~14.5, Proposition 14.6]{Friedlander} that there are natural isomorphisms
\begin{equation}
\label{ettopo}
H^*_{\et,Y}(X,A)\isoto H^*(M(\Et(j)),\Et(U),A).
\end{equation}
When $A=\Z/\ell$, one can use the geometric realization of simplicial sets to let $\cA_{\ell}$ act on the right-hand side of (\ref{ettopo}). We may then let~$\cA_{\ell}$ act on $H^*_{\et,Y}(X,\Z/\ell)$ functorially in $(X,Y)$, using the isomorphism (\ref{ettopo}). By construction, this action has all the properties of the action of $\cA_{\ell}$ on the mod~$\ell$ cohomology of topological spaces (see \cite[\S I.1 and \S VI.1]{ES}). We will use them freely in the sequel.

If $X$ is a complex variety, it follows from \cite[Theorem 8.4]{Friedlander} that Artin's comparison isomorphism $H^q_{\et}(X,\Z/\ell)\isoto H^q(X(\C),\Z/\ell)$ is $\cA_{\ell}$-equivariant.
If $X$ is a variety over $\R$, this statement has a $\Gal(\C/\R)$\nobreakdash-equivariant analogue: the comparison isomorphism $H^q_{\et}(X,\Z/\ell)\isoto H^q_{\Gal(\C/\R)}(X(\C),\Z/\ell)$ of \cite[Corollary 15.3.1]{Scheidererbook} is $\cA_{\ell}$-equivariant as a consequence of \cite[Theorem 1.1]{Cox}.

\subsection{Steenrod operations in twisted \'etale cohomology}
\label{Steenrodtwistedetale}

Guillou and Weibel have explained in \cite[\S\S 1--3]{GW} how to extend the Steenrod operations of \S\ref{Steenrodetale} to \'etale cohomology with twisted coefficients. (They work with the definitions of Epstein \cite{Epstein} and May \cite{May}.) We give a self-contained treatment adapted to our needs,
following their arguments closely.

Consider the morphism of schemes $\pi: S'_{\ell}\to S_{\ell}$, where  $S'_{\ell}:=\Spec(\Z[\frac{1}{\ell},t]/(\frac{t^{\ell}-1}{t-1}))$ and $S_{\ell}:=\Spec(\Z[\frac{1}{\ell}])$. It is a Galois finite \'etale cover with Galois group $\Gamma:=(\Z/\ell)^*$, where $n\in \Gamma$ sends $t$ to~$t^n$.  Decomposing $\Z/\ell[\Gamma]$ as a direct sum of one-dimensional $\Gamma$-representations yields an isomorphism 
\begin{equation}
\label{gradedalgebra}
\pi_*\Z/\ell=\bigoplus_{r\in\Z/(\ell-1)}\bmu_{\ell}^{\otimes r}
\end{equation}
of sheaves of algebras on the big \'etale site of
$S_{\ell}$, where $\bmu_{\ell}^{\otimes r}$ is the factor of $\pi_*\Z/\ell$ on which $n\in \Gamma$ acts by multiplication by $n^r$, and where the algebra structure on the right comes from the canonical isomorphism 
\begin{equation}
\label{canonicaliso}
\Z/\ell\isoto \bmu_{\ell}^{\otimes \ell-1}
\end{equation}
 sending $1$ to $\zeta^{\otimes \ell-1}$ for any primitive $\ell$-th root of unity $\zeta$.

Let $f:Y\hookrightarrow X$ be a closed immersion of 
Noetherian $S_{\ell}$-schemes, and let ${f':Y'\hookrightarrow X'}$ denote its base change by $\pi: S'_{\ell}\to S_{\ell}$.
 As the projection $\pi_X:X'\to X$ is finite, one has $\RR^k\pi_{X,*}\Z/\ell=0$ for $k>0$.
It then follows from (\ref{gradedalgebra}) 
that one has a $\Gamma$\nobreakdash-equivariant decomposition
\begin{equation}
\label{gradedTate}
H_{\et,Y'}^*(X',\Z/\ell)=H^*_{\et,Y}(X,\pi_{X,*}\Z/\ell)=\bigoplus_{r\in\Z/(\ell-1)}H^*_{\et,Y}(X,\bmu_{\ell}^{\otimes r}),
\end{equation}
where $n\in \Gamma$ acts by multiplication by $n^r$ on the $r$-th factor of the right-hand side. 

The action of $\cA_{\ell}$ on the left side of (\ref{gradedTate}) defined in \S\ref{Steenrodetale} commutes with the $\Gamma$\nobreakdash-action by functoriality, and hence preserves the factors of the right side of~(\ref{gradedTate}). Let $\cA_{\ell}$ act on the $H^*_{\et}(X,\bmu_{\ell}^{\otimes r})$ in this way.
For all $\alpha\in\cA_{\ell}$, $x\in H^q_{\et,Y}(X,\bmu_{\ell}^{\otimes r})$ and
$x'\in H^{q'}_{\et}(X,\bmu_{\ell}^{\otimes {r'}})$, the standard properties of Steenrod operations imply that $\alpha(x)\in H^{q+\deg(\alpha)}_{\et,Y}(X,\bmu_{\ell}^{\otimes r})$,
that the identities
\begin{equation}
\begin{alignedat}{5}
\label{identiSq}
\Sq^i(x)&=0& \textrm{ for }&q<i&\textrm{ \hspace{1em}and\hspace{1em} }&\Sq^i(x)=x^2&\textrm{ for }&q=i& \textrm{ \hspace{1em}if }&\ell=2,\\
P^i(x)&=0& \textrm{ for }&q<2i&\textrm{ \hspace{1em}and\hspace{1em} }&P^i(x)=x^{\ell}&\textrm{ for }&q=2i& \textrm{ \hspace{1em}if }&\ell\textrm{ is odd,} 
\end{alignedat}
\end{equation}
are verified, and that the Cartan formulas
\begin{equation}
\begin{alignedat}{5}
\beta(x'\cdot x)&=\beta(x')\cdot x+(-1&)^{q'}x'\cdot\beta(x), \\
\label{Cartan}
\Sq(x'\cdot x)&=\Sq(x')\cdot \Sq(x)&  \textrm{ \hspace{1em}if }&\ell=2,\\
P(x'\cdot x)&=P(x')\cdot P(x)&  \textrm{ \hspace{1em}if }&\ell\textrm{ is odd,} 
\end{alignedat}
\end{equation}
hold in $H^*_{\et,Y}(X,\bmu_{\ell}^{\otimes r+r'})$.

\begin{rems} 
\label{rems}
(i) By construction, the action of $\cA_{\ell}$ on $H^*_{\et,Y}(X,\bmu_{\ell}^{\otimes r})$ only depends on $r\in\Z/(\ell-1)$,  if one uses the canonical isomorphism (\ref{canonicaliso}) to trivialize $\bmu_{\ell}^{\otimes \ell-1}$.  

(ii) We will often tacitly identify Tate twists that are congruent modulo $\ell-1$ in our statements, by means of (\ref{canonicaliso}). 
For instance,  Theorem \ref{compatibilityintro} for $\ell$ odd equates $P^i(\cl(x))\in H^{2i(\ell-1)+2c}_{\et}(X,\bmu_{\ell}^{\otimes c})$ and $\cl(P^i(x))\in H^{2i(\ell-1)+2c}_{\et}(X,\bmu_{\ell}^{\otimes i(\ell-1)+c})$.

(iii) If an isomorphism $\xi:\Z/\ell\isoto\mu_{\ell}^{\otimes r'}$ over $X$ is fixed,  then the induced isomorphisms 
$H^*_{\et,Y}(X,\bmu_{\ell}^{\otimes r})\isoto H^*_{\et,Y}(X,\bmu_{\ell}^{\otimes r+r'})$ are $\cA_{\ell}$-equivariant by (\ref{Cartan})
applied with $x'$ chosen to be the class $\xi\in H^0_{\et}(X,\bmu_{\ell}^{\otimes r'})$ of the fixed isomorphism.

(iv) As a consequence of (iii), if $X$ is a variety over a field $k$ in which a primitive $\ell$-th root of unity $\zeta$ has been fixed,  one obtains $\cA_{\ell}$-equivariant isomorphisms $H^*_{\et,Y}(X,\bmu_{\ell}^{\otimes r})\isoto H^*_{\et,Y}(X,\bmu_{\ell}^{\otimes s})$ for all $r$ and $s$ (which depend on the choice of $\zeta$).
\end{rems}

\subsection{The Bockstein}
\label{parBockstein}

We now proceed to the computation of the Bockstein, following \cite[Proposition 3.3]{GW}.  Let $\delta_r:H^q_{\et,Y}(X,\bmu_{\ell}^{\otimes r})\to H^{q+1}_{\et,Y}(X,\bmu_{\ell}^{\otimes r})$ be the boundary map of the short exact sequence
\begin{equation}
\label{topoBockstein}
0\to\bmu_{\ell}^{\otimes r}\to\bmu_{\ell^2}^{\otimes r}\to\bmu_{\ell}^{\otimes r}\to 0.
\end{equation}
(In \cite{GW}, such morphisms are called $\beta$, but we reserve this notation for the Bockstein element of $\cA_ {\ell}$.)
Use a primitive $\ell$-th root of unity  $\zeta\in\cO(S_{\ell}')$ to identify~$\bmu_{\ell}$ and $\Z/\ell$ on $S_{\ell}'$. The short exact sequence  (\ref{topoBockstein}) for $r=1$ then reads 
\begin{equation}
\label{Bockell}
0\to\Z/\ell\to\bmu_{\ell^2}\to\Z/\ell\to 0,
\end{equation}
and we let $\varpi\in H^1_{\et}(S_{\ell}',\Z/\ell)$ be the image of $1\in H^0_{\et}(S_{\ell}',\Z/\ell)$ by its boundary map. The isomorphism class of the extension (\ref{Bockell}) is independent of the choice of~$\zeta$ (if $n\in (\Z/\ell)^*$, multiplication by $n$ on (\ref{topoBockstein}) induces an isomorphism between short exact sequences (\ref{Bockell}) for choices $\zeta$ and $\zeta^n$ of primitive $\ell$-th roots of unity). It follows that the class $\varpi$ is $\Gamma$-invariant. As the degree of $\pi:S_{\ell}'\to S_{\ell}$ is prime to~$\ell$, the Hochschild--Serre spectral sequence shows that $\varpi$ comes by pull-back from a class that we still denote by $\varpi\in H^1_{\et}(S_{\ell},\Z/\ell)$. 

We still call $\varpi\in H^1_{\et}(X,\Z/\ell)$ the pull-back of $\varpi$ to any $S_{\ell}$-scheme~$X$. When $\ell=2$, the class $\varpi$ is the image of $-1$ by the boundary map of the Kummer short exact sequence. (The class $\varpi$ is called $z$ in \cite{GW}, and $\alpha$ in \cite{Feng} when $\ell=2$.)

\begin{lem}
\label{calculBock}
The Bockstein morphism
 ${\beta:H^q_{\et,Y}(X,\bmu_{\ell}^{\otimes r})\to H^{q+1}_{\et,Y}(X,\bmu_{\ell}^{\otimes r})}$ is equal to $x\mapsto\delta_r(x)-r\,\varpi\cdot x$.
\end{lem}

\begin{proof}
As the degree of the finite \'etale cover $\pi_X:X'\to X$ is prime to $\ell$, the pull-back map $\pi_X^*:H^{*}_{\et,Y}(X,\bmu_{\ell}^{\otimes r})\to H^{*}_{\et,Y'}(X',\bmu_{\ell}^{\otimes r})$ is injective. By functoriality, it thus suffices to prove the lemma on $X'$.  Choose an isomorphism $\Z/\ell\isoto\bmu_{\ell}$ on~$X'$. 
The morphism~$\delta_r$ is then the boundary map of $0\to\Z/\ell\to\bmu_{\ell^2}^{\otimes r}\to\Z/\ell\to 0$, while 
the morphism $\beta$ is the boundary map of 
$0\to\Z/\ell\to\Z/\ell^2\to\Z/\ell\to 0$ (this is one of the standard properties of Steenrod operations).
It now follows from \cite{twistedBock} that $$\delta_r(x)=\beta(x)+r\,\varpi\cdot x.$$ (The reference \cite{twistedBock} deals with the cohomology of topological spaces. By means of the mapping cone construction, the results proven there are still valid for the relative cohomology of pairs of topological spaces. 
By making use of the geometric realization functor, they still hold in the relative cohomology of pairs of (pro-)simplicial spaces.
One can then apply them in our setting using (\ref{ettopo}).)
\end{proof}

\subsection{\'Etale Stiefel--Whitney classes} 
\label{characteristicclasses}

If $f:Y\hookrightarrow X$ is a closed immersion of codimension~$c$ of regular $S_{\ell}$-schemes, the Gysin morphism $\Cl_f:\Z/\ell\to \Rr f^!\bmu_{\ell}^{\otimes c}[2c]$ defined in \cite[D\'efinition 2.3.1]{Riou} is an isomorphism by Gabber's purity theorem \cite[Th\'eor\`eme 3.1.1]{Riou}.  We may then consider the Thom isomorphisms
\begin{equation}
\label{Thomiso}
\phi_{f}:H^q_{\et}(Y,\bmu_{\ell}^{\otimes r})\isoto H^{q+2c}_{\et,Y}(X,\bmu_{\ell}^{\otimes r+c}
)
\end{equation}
induced by $\Cl_f$,  and we let $s_f:=\phi_f(1)\in H^{2c}_{\et,Y}(X,\bmu_{\ell}^{\otimes c})$ be the Thom class.

When $X=E$ is a vector bundle on $Y$ with structural morphism $g:E\to Y$ and $f:Y\hookrightarrow E$ is the zero section,  one has 
\begin{equation}
\label{caracgysin}
\phi_f(x)=g^*(x)\cdot s_f\textrm{\hspace{.5em} for all }x\in H^q_{\et}(Y,\bmu_{\ell}^{\otimes r}).
\end{equation}
In this setting, we define the \'etale Stiefel--Whitney classes $w^{\et}_j(E)\in H^j_{\et}(Y,\Z/\ell)$ of~$E$ by the formulas
\begin{alignat*}{5}
w^{\et}(E)&:=&\,&  \phi_{f}^{-1}(\Sq(s_{f}))&\,&\textrm{ \hspace{.5em}if }\ell=2,\\
w^{\et}(E)&:=&\,& \phi_{f}^{-1}(P(s_{f}))&\,&\textrm{ \hspace{.5em}if }\ell\textrm{ is odd.} 
\end{alignat*}
In view of (\ref{caracgysin}), they are characterized by the identities
\begin{equation}
\label{caracwet}
\Sq(s_f)=g^*w^{\et}(E)\cdot s_f\textrm{ if }\ell=2\textrm{\hspace{.5em} and \hspace{.5em}}P(s_f)=g^*w^{\et}(E)\cdot s_f\textrm{ if }\ell\textrm{ is odd.}
\end{equation}
When $\ell=2$, this definition appears in \cite[p.\,569]{Urabe} for varieties over algebraically closed fields and in \cite[Definition 5.2]{Feng} for varieties over arbitrary fields.
(These classes are denoted by $w_j(E)$ in \cite{Urabe,Feng}, but we want to distinguish them from those defined in \cite{Brosnan}, see (\ref{wBro}).)

The following lemma extends \cite[Lemma 2.6 (4)]{Urabe} and \cite[Lemma~5.7]{Feng}. 
It allows us to define $w^{\et}(\kappa):=w^{\et}(E)\cdot w^{\et}(F)^{-1}$ for any 
class $\kappa=[E]-[F]\in K_0(X)$.

\begin{lem}
\label{mult}
If $Y$ is a regular $S_{\ell}$-scheme and if 
\begin{equation}
\label{sesE}
0\to E_1\to E\to E_2\to 0
\end{equation}
is a short exact sequence of vector bundles on $Y$, one has
\begin{equation}
\label{multwet}
w^{\et}(E)=w^{\et}(E_1)\cdot w^{\et}(E_2).
\end{equation}
\end{lem}

\begin{proof}
Let $\sigma_0,\sigma_1:Y\to \mathbb{A}^1_Y$ be the sections of the projection $\mathbb{A}^1_Y\to Y$ with values~$0$ and $1$. 
Arguing as in \cite[Proof of Lemma 2.7]{Urabe}, one constructs a short exact sequence of vector bundles $0\to F_1\to F\to F_2\to 0$ on $\mathbb{A}^1_Y$ whose pull-backs by $\sigma_0$ and $\sigma_1$ are respectively isomorphic to the split exact sequence ${0\to E_1\to E_1\oplus E_2\to E_2\to 0}$ and to~(\ref{sesE}).
Since the two pull-back morphisms $\sigma_0^*, \sigma_1^*:H^*_{\et}(\mathbb{A}^1_Y,\Z/\ell)\to H^*_{\et}(Y,\Z/\ell)$ are isomorphisms by \cite[XVI, Corollaire~2.2]{SGA43},
one can assume, to prove~(\ref{multwet}), that the exact sequence (\ref{sesE}) is split.

In this case, let $f_1:E_1\to E$ and $f_2:E_2\to E$ be the inclusions of the factors, let $g_1:E\to E_1$ and $g_2:E\to E_2$ be the projections and let $h_1:Y\to E_1$ and ${h_2:Y\to E_2}$ be the zero sections. It then follows from \cite[Remarque~2.3.6]{Riou} and \cite[Proposition 2.3.2]{Riou} that 
$s_f=s_{f_1}\cdot s_{f_2}=g_1^*s_{h_1}\cdot g_2^*s_{h_2}$.
The multiplicativity~(\ref{multwet}) of $w^{\et}$ now follows from the Cartan formulas (\ref{Cartan}).
\end{proof}

If $E$ is a vector bundle on an $S_{\ell}$-scheme $Y$, we let $c^{\et}_j(E)\in H^{2j}_{\et}(Y,\bmu_{\ell}^{\otimes j})$ be the mod~$\ell$ Chern classes of $E$ (see \cite[Th\'eor\`eme 1.3]{Riou} for a construction following \cite{GroChern}).
 We also let $c^{\et}(E):=\sum_{j\geq 0}c^{\et}_j(E)$ be the total mod $\ell$ Chern class of $E$.

The following lemma was proven in \cite[Theorem 5.10]{Feng} when $\ell=2$ and $Y$ is a variety over a finite field. Recall the definition of $\varpi$ given in \S\ref{parBockstein}.

\begin{lem}
\label{wet}
Let $E$ be a vector bundle on a regular $S_{\ell}$-scheme $Y$. Write formally $c^{\et}(E)=\prod_i(1+\lambda_i)$, where the $\lambda_i$ are the Chern roots of $E$. Then 
\begin{alignat}{5}
\label{formulawet2}
w^{\et}(E)&=&\,&\prod_i(1+\varpi+\lambda_i)&\,&\textrm{ \hspace{.5em}if }\ell=2,\\
\label{formulawetodd}
w^{\et}(E)&=&\,&\prod_i(1+\lambda_i^{\ell-1})&\,&\textrm{ \hspace{.5em}if }\ell\textrm{ is odd.} 
\end{alignat}
\end{lem}

\begin{proof}
By Grothendieck's splitting principle, we may assume that $E$ is a successive extension of line bundles.
By Lemma \ref{mult}, we may further assume that $E$ is a line bundle. 
 In this case,  we let $g:E\to Y$ be the structural morphism and $f:Y\hookrightarrow E$ be the zero section.  The image of the Thom class $s_f$ by the morphism $H^2_{\et,Y}(E,\bmu_{\ell})\to H^2_{\et}(E,\bmu_{\ell})$ forgetting the support is equal to 
$c^{\et}_1(\cO_E(Y))=c^{\et}_1(g^*E)=g^*c^{\et}_1(E)$ by \cite[\S 2.1]{Riou}.
As a consequence, 
\begin{equation}
\label{squareThom}
s_f^2=g^*c^{\et}_1(E)\cdot s_f\textrm{ \hspace{.5em}in }H^4_{\et,Y}(E,\bmu_{\ell}^{\otimes 2}).
\end{equation}

If $\ell$ is odd, then $P(s_f)=s_f+s_f^{\ell}$ by (\ref{identiSq}), hence $P(s_f)=(1+g^*c^{\et}_1(E)^{\ell-1})\cdot s_f$ by~(\ref{squareThom}).  Using (\ref{caracwet}), we get $w^{\et}(E)=1+c^{\et}_1(E)^{\ell-1}$,  and (\ref{formulawetodd}) holds.

If $\ell=2$, one has $\Sq(s_f)=s_f+\Sq^1(s_f)+s_f^2$ by (\ref{identiSq}). As $s_f$ lifts to $H^2_{\et,Y}(E,\bmu_{4})$ by \cite[D\'efinition 2.3.1]{Riou},
one has $\Sq^1(s_f)=\varpi\cdot s_f$ by Lemma~\ref{calculBock} and (\ref{topoBockstein}). Applying moreover (\ref{squareThom}) yields $\Sq(s_f)=(1+\varpi+g^*c^{\et}_1(E))\cdot s_f$. We deduce from~(\ref{caracwet}) that $w^{\et}(E)=1+\varpi+c^{\et}_1(E)$,  and hence that (\ref{formulawet2}) holds.
\end{proof}

\subsection{A relative Wu theorem}
\label{parWu}

Let $f:Y\to X$ be a morphism of finite type of regular Noetherian $S_{\ell}$-schemes admitting ample invertible sheaves,
which has virtual relative dimension $-c$ in the sense of \cite[VIII, D\'efinition 1.9]{SGA6}.
As $f$ may be written as the composition of a regular closed immersion $g:Y\to Z$ and of a smooth morphism $h:Z\to X$ (one may take $Z=\P^n_X$ for $n\gg0$), one can define its virtual normal bundle $N_f\in K_0(Y)$ to be $N_f:=-(T_f)^{\vee}$, where $T_f\in K_0(Y)$ is defined in \cite[VIII, Corollaire 2.5]{SGA6}. 
The class $N_f$ has rank $c$. It is equal to $[(I/I^2)^{\vee}]$ if $f$ is a regular closed immersion defined by an ideal $I\subset\cO_{X}$ and to $-[T_{Y/X}]$ if $f$ is smooth.

One defines as in \cite[D\'efinition 2.5.11]{Riou} a morphism $\Cl_f:\Z/\ell\to Rf^!\bmu_{\ell}^{\otimes c}[2c]$ which  coincides with the one considered in \S\ref{characteristicclasses} when $f$ is a closed immersion and which is induced by the trace map when $f$ is smooth \cite[Proposition 2.5.13]{Riou}.
As noted in \cite[Remarque 2.5.14]{Riou}, if $f$ is moreover proper, the morphism $\Cl_f$ induces pushforward (or Gysin) morphisms 
$f_*:H^q_{\et}(Y,\bmu_{\ell}^{\otimes r})\to H^{q+2c}_{\et}(X,\bmu_{\ell}^{\otimes r+c})$.

The following theorem may be deduced from the Riemann--Roch theorems of Panin \cite{Panin} and D\'eglise \cite{Deglise}.
The case of Steenrod squares for varieties over algebraically closed fields also appears in \cite[Proposition~3.1]{ScaSuz2}.

\begin{thm}
\label{thWu}
Let $f:Y\to X$ be a proper morphism of regular Noetherian $S_{\ell}$\nobreakdash-schemes admitting ample invertible sheaves. For $x\in H_{\et}^q(Y,\bmu_{\ell}^{\otimes r})$, one has
\begin{equation}
\begin{alignedat}{5}
\label{idWu}
\Sq(f_*x)&=&\,&f_*(\Sq(x)\cdot w^{\et}(N_f))&\,&\textrm{ \hspace{.5em}if }\ell=2,\\
P(f_*x)&=&\,&f_*(P(x)\cdot w^{\et}(N_f))&\,&\textrm{ \hspace{.5em}if }\ell\textrm{ is odd.}\
\end{alignedat}
\end{equation}
\end{thm}

\begin{proof}
Write $f$ as the composition of a closed immersion $g:Y\hookrightarrow \P^n_X$ and of the natural projection $h:\P^n_X\to X$. Using the covariant functoriality of Gysin morphisms (see \cite[Th\'eor\`eme 2.5.12]{Riou}), the 
 multiplicativity of $w^{\et}$ (see Lemma~\ref{mult}), the Cartan formulas (\ref{Cartan}), and the additivity of $T_f$ \cite[VIII, Corollaire~2.7]{SGA6},
we are reduced to proving the theorem for $g$ and $h$ separately.  We may thus assume that $f$ is either a closed immersion or a projection $\P^n_X\to X$.

Assume first that $f$ is a closed immersion of codimension $c$. Let ${f':Y\to N_{Y/X}}$ and $g':N_{Y/X}\to Y$ be the zero section and the structural morphism of its normal bundle. Let $\sigma_0,\sigma_1:X\to \mathbb{A}^1_X$ be the sections of the projection $\mathbb{A}^1_X\to X$ with values~$0$ and $1$. 
Let $\mathfrak{X}:=\Bl_{\sigma_0(Y)}\A^1_X\setminus\Bl_{\sigma_0(Y)}\sigma_0(X)$ be the deformation of $f$ to its normal cone and $h:\A^1_Y\to\mathfrak{X}$ be the strict transform of the natural closed immersion $\A^1_Y\to\A^1_X$.  One gets a cartesian diagram
\begin{equation*}
\begin{aligned}
\xymatrix@C=1.5em@R=3ex{
Y\ar[r]\ar^{f'}[d]&\A^1_Y\ar^h[d] & Y\ar^f[d]\ar[l] \\
N_{Y/X}\ar[r]& \kX & X\ar[l]\rlap{,}
}
\end{aligned}
\end{equation*}
in which the left (resp.\ right) column is cut out in the middle column by the equation $t=0$ (resp.\ $t=1$), where $t$ is the coordinate of the affine line. The induced diagram
\begin{equation}
\label{diagThoms}
\begin{aligned}
\xymatrix@C=1.5em@R=3ex{
H^*(Y,\bmu_{\ell}^{\otimes r})\ar_{\wr}^{\phi_{f'}}[d]&H^*(\A^1_Y,\bmu_{\ell}^{\otimes r})\ar^{\sim}[r]\ar_{\hspace{1em}\sim}[l]\ar_{\wr}^{\phi_{h}}[d] & H^*(Y,\bmu_{\ell}^{\otimes r})\ar_{\wr}^{\phi_f}[d] \\
H^{*+2c}_{Y}(N_{Y/X},\bmu_{\ell}^{\otimes {r+c}})& H^{*+2c}_{\A^1_Y}(\kX,\bmu_{\ell}^{\otimes {r+c}})\ar[r] \ar[l]&  H^{*+2c}_{Y}(X,\bmu_{\ell}^{\otimes {r+c}})\rlap{,}
}
\end{aligned}
\end{equation}
whose vertical arrows are Thom isomorphisms (\ref{Thomiso}) and whose upper horizontal arrows are isomorphisms by \cite[XVI, Corollaire 2.2]{SGA43} is commutative by \cite[Proposition 2.3.2]{Riou}.
When $\ell=2$, one has 
$$\Sq(\phi_{f'}(x))=g'^*\Sq(x)\cdot\Sq(s_{f'})=g'^*(\Sq(x)\cdot w^{\et}(N_{f'}))\cdot s_{f'}=\phi_{f'}(\Sq(x)\cdot w^{\et}(N_{f'}))$$
in $H^*_Y(N_{Y/X},\bmu_2^{\otimes r+c})$
by (\ref{Cartan}), (\ref{caracgysin}) and (\ref{caracwet}). 
In view of (\ref{diagThoms}) and since $N_f$ and~$N_{f'}$ are both restrictions of $N_h$, we deduce that 
$\Sq(\phi_{f}(x))=\phi_{f}(\Sq(x)\cdot w^{\et}(N_{f}))$ in $H^*_Y(X,\bmu_2^{\otimes r+c})$. Forgetting the support yields (\ref{idWu}). The argument when $\ell$ is odd is identical, replacing $2$ by $\ell$ and $\Sq$ by $P$.

Assume now that $f:\P^n_X\to X$ is the projection. 
Suppose first that $\ell=2$. 
 Let $\lambda:=c^{\et}_1(\mathcal{O}_{\P^n_X}(1))\in H^2_{\et}(\P^n_X,\bmu_{2})$.
As~$\lambda$ lifts to $H^2_{\et}(\P^n_X,\bmu_4)$ by \cite[D\'efinition~2.3.1]{Riou}, Lemma~\ref{calculBock} and~(\ref{topoBockstein}) imply that $\Sq^1(\lambda)=\varpi\cdot \lambda$. Taking (\ref{identiSq}) into account, we get
\begin{equation}
\label{Sqy}
\Sq(\lambda)=\lambda+\varpi\cdot \lambda+\lambda^2.
\end{equation}
The Euler exact sequence $0\to\cO_{\P^n_X}\to \cO_{\P^n_X}(1)^{\oplus N+1}\to T_{\P^n_X/X}\to 0$ and Lemmas~\ref{mult} and \ref{wet} imply that $w^{\et}(T_{\P^n_X/X})=(1+\varpi+\lambda)^{n+1}/(1+\varpi)$ hence that 
\begin{equation}
\label{Nf}
w^{\et}(N_f)=w^{\et}(-[T_{\P^n_X/X}])=(1+\varpi)/(1+\varpi+\lambda)^{n+1}.
\end{equation}
As $H^*_{\et}(\P^n_X,\bmu_2^{\otimes *})=H^*_{\et}(X,\bmu_2^{\otimes *})[\lambda]/\lambda^{n+1}$ by the projective bundle formula (see for instance \cite[\S 1]{Riou}), it suffices to prove the identity (\ref{idWu}) for $x$ of the form $f^*y\cdot \lambda^m$, where $y\in H_{\et}^*(X,\bmu_2^{\otimes *})$ and $m\in\{0,\dots, n\}$. By the projection formula, the left-hand side of (\ref{idWu}) is equal to $\Sq(y)$ if $m=n$ and vanishes otherwise. Let us compute the right-hand side of (\ref{idWu}). One has 
$$\Sq(x)=\Sq(f^*y\cdot \lambda^m)=f^*\Sq(y)\cdot\Sq(\lambda)^m=f^*\Sq(y)\cdot \lambda^m(1+\varpi+\lambda)^m$$
by the Cartan formula (\ref{Cartan}) and (\ref{Sqy}). It follows from (\ref{Nf}) that 
\begin{equation}
\label{expansion}
\Sq(x)\cdot w^{\et}(N_f)=f^*\Sq(y)\cdot \lambda^m\eta(\eta+\lambda)^{-k-1}=f^*\Sq(y)\cdot \lambda^m\eta^{m-n}(1+\frac{\lambda}{\eta})^{-k-1},
\end{equation}
where $\eta:=1+\varpi$ and $k:=n-m$. By Lemma \ref{powerseries} below applied with $T=\frac{\lambda}{\eta}$, the coefficient of $\lambda^n$ in (\ref{expansion}) is equal to $f^*\Sq(y)$ if $m=n$ and to $0$ otherwise. It then follows from the projection formula that the right-hand side of (\ref{idWu}) is equal to $\Sq(y)$ if $m=n$ and vanishes otherwise. This concludes the proof when $\ell=2$.

The argument when $\ell$ is odd is similar, only easier because $\varpi$ is not involved. Equations (\ref{Sqy}) and (\ref{Nf}) are replaced with the identities $P(\lambda)=\lambda+\lambda^{\ell}$ and $w^{\et}(N_f)=1/(1+\lambda)^{n+1}$. Both the left and the right side of (\ref{idWu}) for $x=f^*y\cdot \lambda^m$ with $y\in H_{\et}^{q-2m}(\P^n_X,\bmu_{\ell}^{\otimes r-m})$ are then computed to be equal to $P(y)$ if $m=n$ and to vanish otherwise (using Lemma \ref{powerseries} applied with $T=\lambda$).
\end{proof}

\begin{lem}
\label{powerseries}
The coefficient of $T^k$ in the power series $F(T)=(1+T)^{-k-1}\in\Z[[T]]$ is odd if $k=0$ and even if $k\geq 1$.
\end{lem}

\begin{proof}
Since $F^{(k)}(T)=(-1)^k\prod_{i=k+1}^{2k} i \cdot(1+T)^{-2k-1}$, the coefficient of~$T^k$ in $F(T)$ is 
$\frac{F^{(k)}(0)}{k!}=(-1)^k\binom{2k}{k}$. It is equal to $1$ if $k=0$ and it is even if $k\geq 1$.
\end{proof}

\begin{rem}
 In topology, the Wu theorem describes the interaction of Steenrod operations and Poincar\'e duality. The same holds for Theorem \ref{thWu}, where Poincar\'e duality appears in the guise of pushforwards. Feng's Wu theorem over finite fields \cite[Theorem 6.5]{Feng}, however, describes the interaction of Steenrod operations and a subtler duality that combines Poincar\'e duality and duality in Galois cohomology of finite fields. It is therefore not a particular case of Theorem \ref{thWu}.
\end{rem}

\section{Steenrod operations and algebraic classes}

We now apply the relative Wu theorem proved in Section \ref{secWu} to study the compatibility of Brosnan's Steenrod operations on Chow groups (recalled in \S\ref{parSteenrodChow}) and those in \'etale cohomology (defined in \S\S\ref{Steenrodetale}--\ref{Steenrodtwistedetale}). Our main result in this direction is Theorem \ref{compatibility}. In \S\ref{parobstr}, we explain how to deduce obstructions to the algebraicity of cohomology classes. We clarify their relation with complex cobordism in~\S\ref{parMU}.

Throughout this section, we fix a field $k$ and a prime number $\ell$ invertible in $k$.

\subsection{Steenrod operations on Chow groups}
\label{parSteenrodChow}

Let $X$ be a smooth 
 variety over~$k$.
Brosnan has defined in \cite[Definition~8.11 and \S 11]{Brosnan} an action of $\cA_{\ell}$ on $\CH^*(X)/\ell$ such that all odd degree elements of $\cA_{\ell}$ act by zero, and such that for all $\alpha\in\cA_{\ell}$ of even degree and all $x\in \CH^c(X)/\ell$, one has $\alpha(x)\in \CH^{c+\deg(\alpha)/2}(X)/\ell$. 

Let $E$ be a vector bundle on a smooth variety $X$ over~$k$. Let $c_j(E)\in \CH^j(X)/\ell$ be the mod $\ell$ Chern classes of $E$ defined as in \cite{GroChern} (see also \cite[p.\,325]{Fulton}),
 and let $c(E):=\sum_{j\geq 0} c_j(E)$. Write formally $c(E)=\prod_i(1+\lambda_i)$, where the $\lambda_i$ are the Chern roots of $E$. 
Brosnan introduces in \cite[p.\,1891]{Brosnan} the characteristic classes $w_{2j}(E)\in \CH^j(X)/\ell$ defined by
\begin{equation}
\label{wBro}
w(E)=\sum_{j\geq 0} w_{2j}(E):=\prod_i (1+\lambda_i^{\ell-1}).
\end{equation}
This definition is extended to $K_0(X)$ by setting $w([E]-[F]):=w(E)\cdot w(F)^{-1}$.

Brosnan uses these classes to describe the behaviour of the Steenrod operations on Chow groups with respect to proper pushforwards.

\begin{prop}
\label{WuChow}
Let $f:Y\to X$ be a proper morphism between smooth varieties over $k$ with virtual normal bundle $N_f$. For $x\in \CH^*(Y)/\ell$, one has
\begin{equation*}
\begin{alignedat}{5}
\Sq(f_*x)&=&\,&f_*(\Sq(x)\cdot w(N_f))&\,&\textrm{ \hspace{.5em}if }\ell=2,\\
P(f_*x)&=&\,&f_*(P(x)\cdot w(N_f))&\,&\textrm{ \hspace{.5em}if }\ell\textrm{ is odd.}\
\end{alignedat}
\end{equation*}
\end{prop}

\begin{proof}
This follows from \cite[Propositions 9.4 (iii) and 10.3]{Brosnan}.
\end{proof}

\begin{rem}
\label{RemVoe}
An earlier construction of Steenrod operations on Chow groups was given by Voevodsky in the context of motivic cohomology \cite{Vo2}. We do not know a published proof of the fact that his operations are compatible with Brosnan's.
\end{rem}

\subsection{Comparing Steenrod operations}

Combining Proposition \ref{WuChow} and Theorem \ref{thWu}, we may compare the Steenrod operations on Chow groups and in \'etale cohomology. To do so, we let $\cl:\CH^*(X)/\ell\to H^{2*}_{\et}(X,\bmu_{\ell}^{\otimes *})$ denote the mod $\ell$ \'etale cycle class map  of a smooth variety $X$ over $k$. 

\begin{lem}
\label{wetChow}
Let $X$ be a smooth variety over $k$ and let $\kappa\in K_0(X)$ be a $\KK$-theory class of rank $c\in\Z$. Then
\begin{equation*}
\begin{alignedat}{5}
w^{\et}(\kappa)&=&\,&\sum_{j\geq 0}(1+\varpi)^{c-j}\cdot\cl(w_{2j}(\kappa))&\,&\textrm{ \hspace{.5em}if }\ell=2,\\
w^{\et}(\kappa)&=&\,&\cl(w(\kappa))&\,&\textrm{ \hspace{.5em}if }\ell\textrm{ is odd.}\
\end{alignedat}
\end{equation*}
\end{lem}

\begin{proof}
It follows from Lemma \ref{mult} and the definition (\ref{wBro}) that both sides of these equalities are multiplicative in the class $\kappa$.  By the splitting principle, we may thus assume that $\kappa$ is the class of a line bundle $L$. It then follows from Lemma \ref{wet} and~(\ref{wBro}) that, if $\ell$ is odd, one has
$$w^{\et}(\kappa)=1+c_1^{\et}(L)^{\ell-1}=\cl(1+c_1(L)^{\ell-1})=\cl(w(\kappa)),$$
and that, if $\ell=2$,  one has
$$w^{\et}(\kappa)=1+\varpi+c_1^{\et}(L)=1+\varpi+\cl(c_1(L))=(1+\varpi)\cdot\cl(w_0(\kappa))+\cl(w_2(\kappa)).\eqno\qedhere$$
\end{proof}

\begin{thm}
\label{compatibility}
Let $X$ be a smooth quasi-projective variety over $k$. For $c\geq 0$ and $x\in \CH^c(X)/\ell$, one has
\begin{alignat*}{5}
\Sq(\cl(x))&=&\,&\sum_{i\geq 0}(1+\varpi)^{c-i}\cdot\cl(\Sq^{2i}(x))&\,&\textrm{ \hspace{.5em}if }\ell=2,\\
\nonumber P(\cl(x))&=&\,&\cl(P(x))&\,&\textrm{ \hspace{.5em}if }\ell\textrm{ is odd.}\
\end{alignat*}
\end{thm}

\begin{proof}
By invariance of \'etale cohomology under inseparable extensions of the base field \cite[VIII, Corollaire 1.2]{SGA42}, we may assume that $k$ is perfect (after replacing it with its perfect closure). We may assume that $x$ is the class of integral subvariety $g:Z\hookrightarrow X$. As $k$ is perfect, it follows from Gabber's improvement on de Jong's alteration theorem \cite[Theorem 2.1]{IlluTem} that there exists a smooth variety $Y$ over~$k$ and a projective morphism $\nu:Y\to Z$ that is generically finite of degree $d$ prime to~$\ell$.  Set $f:=g\circ\nu$.  Replacing $x$ by $d\;\!x$,  which is legitimate because $d\in(\Z/\ell)^*$, we may assume that $x=f_*[Y]$.  If $\ell$ is odd,  we then compute 
\begin{equation*}
\begin{alignedat}{5}
P(\cl(x))&=P(\cl(f_*[Y]))=P(f_*1)=f_*w^{\et}(N_f)\hspace{1em}\textrm{and}\\
\cl(P(x))&=\cl(P(f_*[Y]))=\cl(f_*(P([\wY])\cdot w(N_f)))=\cl(f_*w(N_f))=f_*\cl(w(N_f)),
\end{alignedat}
\end{equation*}
where the third equality is Theorem \ref{thWu} and the fifth is Proposition \ref{WuChow}.
If $\ell=2$, the same arguments show that
$\Sq(\cl(x))=f_*w^{\et}(N_f)$ and $\cl(\Sq(x))=f_*\cl(w(N_f))$.
In both cases, the theorem now follows from Lemma \ref{wetChow}.
\end{proof}

One can complement Theorem \ref{compatibility} with a formula describing the action of the Bockstein on a cycle class. 

\begin{lem} 
\label{compaBock}
For  a smooth quasi-projective variety $X$ over $k$ and $x\in \CH^c(X)/\ell$,
\begin{equation*}
\beta(\cl(x))=-c\,\varpi\cdot\cl(x).
\end{equation*}
\end{lem}
\begin{proof}
This follows from Lemma \ref{calculBock} because $\cl(x)$ lifts to $H^{2c}_{\et}(X,\bmu_{\ell^2}^{\otimes c})$.
\end{proof}

\begin{rem}
\label{RemVoe2}
We do not think that Theorem \ref{compatibility} was previously known, even in the framework of Voevodsky's motivic Steenrod operations alluded to in Remark \ref{RemVoe}. We note for instance that the class $\varpi$ does not appear in the statement of \cite[Theorem~1.1~(iii)]{BJ}. We also note that the claim made in \cite[Remark~7.3.1]{GW} implies that the Steenrod operations in \'etale cohomology considered in \cite{GW} preserve algebraic classes, which is not true in general (see Remark \ref{remSteencl} below).
\end{rem}

\subsection{Obstructions to algebraicity}
\label{parobstr}

Let $H^{2*}_{\et}(X,\bmu_{\ell}^{\otimes *})_{\alg}\subset H^{2*}_{\et}(X,\bmu_{\ell}^{\otimes *})$ be the image of the  cycle class map. We obtain stability results of these algebraic classes under Steenrod operations.

\begin{prop}
\label{stability}
Let $X$ be a smooth quasi-projective variety over $k$. 
\begin{enumerate}[(i)]
\item If $k$ contains a primitive $\ell^2$-th root of unity,  then $H^{2*}_{\et}(X,\bmu_{\ell}^{\otimes *})_{\alg}$ is  $\cA_{\ell}$-stable.
\item If $\ell$ is odd, then $H^{2*}_{\et}(X,\bmu_{\ell}^{\otimes *})_{\alg}$ is stable under the action of $P$.
\end{enumerate}
\end{prop}

\begin{proof}
By construction, $\varpi\in H^1_{\et}(\Spec(k),\Z/\ell)$ vanishes if $k$ contains a primitive $\ell^2$\nobreakdash-th root of unity. Theorem \ref{compatibility} and Lemma~\ref{compaBock} thus imply the proposition.
\end{proof}

\begin{rem}
\label{remSteencl}
When $\ell=2$, the hypothesis that $-1$ is a square in Proposition \ref{stability}~(i) cannot be removed, even if one restricts to even degree Steenrod squares. 
For instance, if $k=\R$ and $X=\P^2_{\R}$, then 
$$H^*_{\et}(X,\Z/2)=H^*_{\et}(\Spec(\R),\Z/2)[\lambda]/\lambda^3=\Z/2[\varpi,\lambda]/\lambda^3,$$ where $\lambda:=c^{\et}_1(\cO_X(1))$. The class $\lambda$ is algebraic, hence so is $\lambda^2$.
Using (\ref{identiSq}) and~(\ref{Cartan}) shows that  $\Sq^2(\lambda^2)=\Sq^1(\lambda)^2$. Since $\lambda$ lifts to $H^2_{\et}(X,\bmu_4)$, Lemma \ref{calculBock} implies that $\Sq^1(\lambda)=\varpi\cdot\lambda$. It follows that 
$\Sq^2(\lambda^2)=\varpi^2\cdot\lambda^2$ is not algebraic.
This shows that Steenrod operations do not preserve algebraic classes in general. 

In this example, the  algebraic class $\lambda^2$ is not sent to an algebraic class by $\Sq^2$, but it is sent to an algebraic class by $\Sq^2+\varpi^2$. This is a general phenomenon, as we show in Proposition \ref{weirdobstr} (i) below.
\end{rem}

Traditionally, instances of Proposition \ref{stability} have been used to prove that some cohomology classes are not algebraic, by showing that they are not killed by an odd degree Steenrod operations.

\begin{cor}
\label{basicobstr}
Let $X$ be a smooth quasi-projective variety over $k$. 
Assume that~$k$ contains a primitive $\ell^2$-th root of unity. Then $H^{2c}_{\et}(X,\bmu_{\ell}^{\otimes c})_{\alg}$ is killed by all odd degree Steenrod operations.
\end{cor}

Even over algebraically closed fields, Proposition \ref{stability} yields obstructions to the algebraicity of cohomology classes that go beyond the vanishing of odd degree Steenrod operations.
Over the complex numbers,  that algebraic mod $\ell$ cohomology classes are reductions mod $\ell$ of integral Hodge classes implies the following.

\begin{cor}
\label{Hodgeobstr}
Let $X$ be a smooth projective variety over $\C$. Fix $\alpha\in\cA_{\ell}$ and $x\in H^{2c}(X(\C),\Z/\ell)_{\alg}$.
\begin{enumerate}[(i)]
\item If $\deg(\alpha)$ is odd, then $\alpha(x)=0$.
\item If $\deg(\alpha)$ is even, then $\alpha(x)$ is the reduction mod $\ell$ of an integral Hodge class.
\end{enumerate}
\end{cor}

As algebraic classes are reductions mod $\ell$ of Tate classes, we obtain similarly:

\begin{cor}
\label{Tateobstr}
Assume that $k$ is the algebraic closure of a finitely generated field. Let $X$ be a smooth projective variety over $k$.  Fix $\alpha\in\cA_{\ell}$ and $x\in H^{2c}_{\et}(X,\bmu_{\ell}^{\otimes c})_{\alg}$.
\begin{enumerate}[(i)]
\item If $\deg(\alpha)$ is odd, then $\alpha(x)=0$.
\item If $\deg(\alpha)$ is even, then $\alpha(x)$ is the reduction mod $\ell$ of an integral Tate class.
\end{enumerate}
\end{cor}

Corollary \ref{Hodgeobstr} (i) and \ref{Tateobstr} (i) are the well-known topological obstructions to the integral Hodge and Tate conjectures going back to Atiyah and Hirzebruch \cite{AH}.
Corollary~\ref{Hodgeobstr}~(ii) and \ref{Tateobstr} (ii) yield new Hodge-theoretic (or Galois-theoretic) obstructions to them.

When $k$ does not contain a primitive $\ell^2$-th root of unity, one can still deduce from Theorem \ref{compatibility} and Lemma~\ref{compaBock} stability properties of algebraic cohomology classes by Steenrod 
operations, and hence obstructions to the algebraicity of cohomology classes that go beyond Proposition
\ref{stability} and Corollary \ref{basicobstr}. We give an example of such an obstruction, to be used in \S\ref{cexfinite} and \S\ref{cexreal},  in the next proposition. 

\begin{prop}
\label{weirdobstr}
Assume that $\ell=2$ and let $X$ be a smooth quasi-projective variety over $k$.  Then
\begin{enumerate}[(i)]
\item $\Sq^2+\frac{c(c-1)}{2}\varpi^2$ sends $H^{2c}_{\et}(X,\bmu_{2}^{\otimes c})_{\alg}$ to $H^{2c+2}_{\et}(X,\bmu_{2}^{\otimes c+1})_{\alg}$;
\item $\Sq^3+(c+1)\,\varpi\cdot\Sq^2+\frac{c(c-1)}{2}\,\varpi^3$ kills $H^{2c}_{\et}(X,\bmu_{2}^{\otimes c})_{\alg}$.
\end{enumerate} 
\end{prop}

\begin{proof}
If $x\in \CH^c(X)/2$, it follows from Theorem \ref{compatibility} that 
$$\cl(\Sq^2(x))=\Sq^2(\cl(x))+\frac{c(c-1)}{2}\varpi^2\cdot\cl(x).$$
This proves (i). 
 Applying Lemma~\ref{compaBock}, we deduce that $H^{2c}_{\et}(X,\bmu_{2}^{\otimes c})_{\alg}$ is killed by $(\Sq^1+(c+1)\,\varpi)(\Sq^2+\frac{c(c-1)}{2}\varpi^2)$. Developing and applying Lemma \ref{compaBock} again, we see that
this operator equals $\Sq^3+(c+1)\,\varpi\cdot\Sq^2+\frac{c(c-1)}{2}\,\varpi^3$, proving (ii).
\end{proof}

\subsection{Relation with complex cobordism}
\label{parMU}

In \cite{Totarobordism}, Totaro reinterpreted the Atiyah--Hirzebruch topological obstructions to the integral Hodge conjecture using complex cobordism. His point of view is that algebraic cohomology classes on a smooth projective complex  variety $X$ always lie in the image of the natural map ${\mu:\MU^*(X)\to H^*(X(\C),\Z)}$. Consequently, an integral Hodge class not in the image of $\mu$ cannot be algebraic. This fact also explains the topological obstructions to algebraicity given by Corollary \ref{Hodgeobstr} (i) (see \cite[Proposition 3.6]{BOConiveau}). 
We now explain how to similarly understand the obstruction of Corollary \ref{Hodgeobstr} (ii) using complex cobordism.

Let $\mathbf{MU}$ and $\mathbf{H}(\Z)$ be spectra representing complex cobordism and ordinary cohomology with integral coefficients. The morphism $\mu: \MU^*(X)\to H^*(X(\C),\Z)$ is induced by a map of spectra $\mu:\mathbf{MU}\to\mathbf{H}(\Z)$. Let $\mathbf{H}(\pi_*(\mathbf{MU})_\C)$ denote the Eilenberg--MacLane spectrum of the graded ring $\pi_*(\mathbf{MU})_\C:=\pi_*(\mathbf{MU})\otimes_{\Z}\C$. In \cite[\S 5]{HQ}, Hopkins and Quick define a map $\iota:\mathbf{MU}\to \mathbf{H}(\pi_*(\mathbf{MU})_\C)$ inducing the complexification $\pi_*(\mathbf{MU})\to \pi_*(\mathbf{MU})_\C$ on homotopy groups. If $X$ is a smooth projective complex variety of dimension $n$ and $c\in\Z$, the map $\iota$ induces morphisms 
 \begin{equation}
\label{MUHodge}
\iota_*:\MU^{2c}(X)\to \bigoplus_{e=0}^{n-c} H^{2c+2e}(X(\C),\pi_{2e}(\mathbf{MU})_{\C}).
\end{equation}
One recovers the morphism $\mu:\MU^{2c}(X)\to H^{2c}(X(\C),\C)$ by projecting (\ref{MUHodge}) onto the $e=0$ summand. 
 In \cite[(17)]{HQ}, Hopkins and Quick define $\Hdg_{\MU}^{2c}(X)$ to be the set of $y\in \MU^{2c}(X)$ such that all the components of $\iota_*y$ are Hodge. They prove that any algebraic class $x\in H^{2c}(X(\C),\Z)$ belongs to $\mu(\Hdg_{\MU}^{2c}(X))$ (see \cite[Corollary~7.12]{HQ}). We claim that this obstruction to algebraicity combining Hodge theory and complex cobordism is finer than Corollary \ref{Hodgeobstr}.

\begin{prop}
Let $X$ be a smooth projective complex variety and let $\ell$ be a prime number. Fix $\alpha\in\cA_{\ell}$ and let $x$ be a class in the image of the composition
\begin{equation*}
\label{composition}
\Hdg^{2c}_{\MU}(X)\xrightarrow{\mu} H^{2c}(X(\C),\Z)\xrightarrow{\rho} H^{2c}(X(\C),\Z/\ell)
\end{equation*}
of $\mu$ and of the reduction modulo $\ell$ morphism $\rho$.
\begin{enumerate}[(i)]
\item If $\deg(\alpha)$ is odd, then $\alpha(x)=0$.
\item If $\deg(\alpha)$ is even, then $\alpha(x)$ is the reduction mod $\ell$ of an integral Hodge class. 
\end{enumerate}
\end{prop}

\begin{proof}
Assertion (i) follows from \cite[Proposition 3.6]{BOConiveau} (the argument given there when $\ell=2$ works as well when $\ell$ is odd). We now prove (ii).

Landweber \cite[Theorem 8.1]{Landweber} has shown the existence of a
stable cohomological operation $\talpha$ in complex cobordism such that $(\rho\circ \mu)\circ\talpha=\alpha\circ (\rho\circ\mu)$. Letting~$2d$ denote the degree of $\alpha$, we get a diagram of spectra
 \begin{equation}
\label{spectra}
\begin{aligned}
\xymatrix
@R=0.4cm
{
 \mathbf{H}(\pi_*(\mathbf{MU})_\C)\ar^{\pi_*(\talpha)_\C\hspace{1em}}
[r]&\Sigma^{2d}\mathbf{H}(\pi_*(\mathbf{MU})_\C)
\\
\mathbf{MU}\ar^{\talpha\hspace{.9em}}[r]\ar^{\iota}[u]\ar_{\rho\circ\mu}[d]& \Sigma^{2d}\mathbf{MU}\ar^{\iota}[u]\ar_{\rho\circ\mu}[d] \\
 \mathbf{H}(\Z/\ell)\ar^{\alpha\hspace{1em}}[r]&\Sigma^{2d}\mathbf{H}(\Z/\ell)
}
\end{aligned}
\end{equation}
where $\mathbf{H}(\Z/\ell)$ is the Eilenberg--MacLane spectrum with coefficients $\Z/\ell$. We claim that (\ref{spectra}) is commutative up to homotopy. The lower square commutes by our choice of $\talpha$. For the upper square, we need to show that $\iota\circ\talpha$ and $\pi_*(\talpha)_{\C}\circ\iota$ classify the same element of $H^*(\mathbf{MU},\pi_{*-2d}(\mathbf{MU})_{\C})=\Hom(H_*(\mathbf{MU},\Z),\pi_{*-2d}(\mathbf{MU})_{\C})$.
As the Hurewicz morphism $\pi_*(\mathbf{MU})\to H_*(\mathbf{MU},\Z)$ is injective by \cite[II, Corollary 8.11]{Adams}, it suffices to verify that the maps $\iota\circ\talpha$ and $\pi_*(\talpha)_{\C}\circ\iota$ induce the same morphism $\pi_*(\mathbf{MU})\to \pi_{*-2d}(\mathbf{MU})_{\C}$, which is tautological.

 Now let $y\in \Hdg^{2c}_{\MU}(X)$ be such that $x=\rho(\mu(y))$.
 The upper square of (\ref{spectra}) shows that $\talpha(y)\in \Hdg^{2c+2d}_{\MU}(X)$. 
By the lower square of~(\ref{spectra}), the class $\alpha(x)$ is the reduction mod $\ell$ of the integral Hodge class $\mu(\talpha(y))$.
\end{proof}

It should be possible to carry out a similar analysis in the setting of the integral Tate conjecture, by relying on Quick's work \cite{Quickcobordism,Quicktorsion}. We do not do it here.

\section{Non-algebraic cohomology classes}

In this section, we construct new examples of non-algebraic cohomology classes, over algebraically closed fields (in \S \ref{par5} and \S\ref{partf}), finite fields (in \S\ref{cexfinite}) and the field~$\R$ of real numbers (in \S\ref{cexreal}). Applications to unramified cohomology are made explicit in \S\ref{unramified}.  In addition, we review in \S\ref{parapprox} the construction of algebraic approximations of classifying spaces that we use in \S\ref{partf} and \S\ref{cexfinite}.

We denote by $\rho$ all the reduction modulo $\ell$ morphisms: in Betti cohomology,  in equivariant Betti cohomology, or in $\ell$-adic \'etale cohomology, where the prime number $\ell$ is always clear from the context (and equal to $2$ in \S\ref{par5} and \S\ref{cexreal}).

\subsection{A fivefold failing the integral Tate conjecture \texorpdfstring{over $\oF$}{}}
\label{par5}

\begin{prop}
\label{prop5}
There exists a smooth projective complex variety $Y$ such that:
\begin{enumerate}[(i)]
\item The variety $Y$ has dimension $5$.
\item There exists $y\in H^4(Y(\C),\Z)[2]$ such that $\Sq^3(\rho(y))\neq 0$.
\item The variety $Y$ is defined over $\overline{\Q}$, with good reduction at all places not above~$2$.
\end{enumerate}
\end{prop}

\begin{proof}
Let $Z$ be the complex fourfold defined in \cite[Proposition 5.3]{BOConiveau}. It carries a class $\sigma\in H^3(Z(\C),\Z)[2]$ such that $\rho(\sigma)^2\neq 0$ in $H^6(Z(\C),\Z/2)$. Let $E$ be a complex elliptic curve. Choose $\tau\in H^1(E(\C),\Z)$ not divisible by $2$. Set $Y:=Z\times E$ and $y:=p_1^*\sigma\cdot p_2^*\tau$. Then Cartan's formula and the vanishing of $\Sq^1(\rho(\tau))$ imply $$\Sq^3(\rho(y))=p_1^*\Sq^3(\rho(\sigma))\cdot p_2^*\rho(\tau)=p_1^*\rho(\sigma)^2\cdot p_2^*\rho(\tau)\neq 0.$$

That $Y$ may be chosen to have good reduction at odd places was explained in \cite[\S 4.1]{ScaSuz}. To give more details,  choose the elliptic curve $E$, as well as all the elliptic curves appearing in the construction of $Z$ given in \cite[\S5.1]{BOConiveau} to have complex multiplication. Such elliptic curves are defined over a number field and have potentially good reduction everywhere. Combined with the fact that degree~$2$ finite \'etale covers of smooth projective varieties specialize in characteristic not $2$ (see \cite[X, Th\'eor\`eme~3.8]{SGA1}),
 this shows that the construction of $Z$, hence also of $Y$, works in mixed characteristic away from characteristic~$2$. The variety $Y$ is thus defined over $\overline{\Q}$ with good reduction away from~$2$.
\end{proof}

\begin{thm}
\label{th1b}
Let $p$ be an odd prime number. There exists a smooth projective variety $X$ of dimension $5$ over $\oF$ on which the $2$-adic integral Tate conjecture for codimension $2$ cycles fails.
\end{thm}

\begin{proof}
Let $Y$ and $y$ be as in Proposition \ref{prop5}. Let $X$ denote a smooth projective variety over~$\oF$ obtained by reduction of $Y$ modulo $p$. Artin's comparison theorem \cite[XI, Th\'eor\`eme 4.4]{SGA43}, the invariance of \'etale cohomology under extensions of algebraically closed fields
\cite[XVI, Corollaire 1.6]{SGA43}, and the smooth and proper base change theorems 
\cite[XVI, Corollaire 2.2]{SGA43} yield a commutative diagram
\begin{equation}
\label{red1}
\begin{aligned}
\xymatrix
@R=0.3cm
{
H^*_{\et}(X,\Z_2)\ar^{\sim\hspace{.6em}}[r]\ar^{\rho}[d]&
H^*(Y(\C),\Z_2)\ar^{\rho}[d]  \\
 H^*_{\et}(X,\Z/2)\ar^{\sim\hspace{.6em}}[r]&H^*(Y(\C),\Z/2)
}
\end{aligned}
\end{equation}
whose horizontal arrows are isomorphisms, and whose bottom row is $\cA_2$-equivariant (by functoriality of Steenrod operations and $\cA_2$-equivariance of Artin's comparison isomorphism, see \S\ref{Steenrodetale}). 
The class $x\in H^4_{\et}(X,\Z_2(2))[2]$ induced by~$y$ and a fixed isomorphism $\Z_2\isoto\Z_2(2)$ over $\oF$ is Tate since it is torsion. As $\Sq^3(\rho(x))\neq 0$ by Proposition \ref{prop5} (ii) and (\ref{red1}), the class $x$ is not algebraic by Corollary~\ref{Tateobstr}~(i).
\end{proof}

\begin{rem}
\label{rem1}
The same argument shows that the complex variety $Y$ of Proposition~\ref{prop5} is a $5$\nobreakdash-dimensional counterexample to the integral Hodge conjecture: the class $y$ is Hodge because it is torsion, and it is not algebraic by Corollary \ref{Hodgeobstr} (i).
\end{rem}

\subsection{Algebraic approximations of classifying spaces}
\label{parapprox}

Proposition \ref{approx} below is essentially due to Ekedahl \cite[Theorem 1.3 and~\S 1.1]{Ekedahl}. Simplifications to the proof were made by Pirutka and Yagita \mbox{\cite[\S 4]{PY}} with further additions by Antieau \cite[\S 3]{Antieau}. We give a proof for the convenience of the reader, and claim no originality.

Recall that one may let any Lie group $H$ act freely and properly on a contractible space $EH$. The resulting quotient $BH:=EH/H$ is the \textit{classifying space} of $H$. 

We also recall that a continuous map $f:Y\to X$ between topological spaces is an $N$-\textit{equivalence} if $f_*:\pi_q(Y,y)\to\pi_q(X,f(y))$ is an isomorphism for $q<N$ and a surjection for $q=N$, for all $y\in Y$.
By the Hurewicz theorem \cite[Theorem 4.37]{Hatcher} and the universal coefficient theorem,  the morphism $f^*:H^q(X,A)\to H^q(Y,A)$ is then an isomorphism for $q<N$ and an injection for $q=N$, for all abelian groups~$A$.

\begin{prop}
\label{approx}
Let $H$ be a connected
 reductive complex Lie group. 
Let $p$ be a prime number, and fix $N\geq 1$.
Then there exists a smooth projective complex variety $Y$ of dimension $N$ such that:
\begin{enumerate}[(i)]
\item There exists an $N$-equivalence $Y(\C)\to B(H\times \C^*)$.
\item The variety $Y$ is defined over $\Q$ and has good reduction at $p$.
\end{enumerate}
\end{prop}

The following lemma is \cite[Lemma 1.1 (ii)-(iv)]{Ekedahl}.

\begin{lem}
\label{eke}
Let $V$ be a $d$-dimensional vector space over an algebraically closed field $K$. For $n\geq d$, let $U_n\subset V^{\oplus n}$ be the open subset of $n$-tuples that span $V$. 
\begin{enumerate}[(i)]
\item The complement of $U_n$ in $V^{\oplus n}$ has codimension $\geq n+1-d$.
\item The algebraic group $\GL(V)$ acts freely on $U_n$.
\item The image of $U_n$ in $\P(V^{\oplus n})$ consists of $\SL(V)$-stable points.
\end{enumerate}
\end{lem}

\begin{proof}[Proof of Proposition \ref{approx}]
Let $\cH$ be a reductive group scheme over $\Z_{(p)}$ (in the sense of \cite[XIX, D\'efinition 2.7]{SGA33r}) with $\cH(\C)\simeq H$ (see \cite[XXV, Corollaire~1.3]{SGA33r}). By \cite[VI$_B$, Proposition 13.2]{SGA31r}, there exists a closed embedding of group schemes $\cH\hookrightarrow \GL_{d-2,\Z_{(p)}}$ for some $d\geq2$. Composing it with
\begin{equation*}
M\mapsto
\label{matrix}
\bigg(\begin{smallmatrix} 
M & 0 & 0\\
0 & \det(M)^{-1}  &0 \\ 
 0&0 & 1
\end{smallmatrix}\bigg)
\end{equation*}
yields embeddings $\cH\hookrightarrow \SL_{d,\Z_{(p)}}$ and $\cH\times\bG_m \hookrightarrow \GL_{d,\Z_{(p)}}$ (where~$\bG_m $ is embedded by homotheties). Fix $n\gg 0$. Let $\cH\times\bG_m$ act linearly on $\A^{nd}_{\Z_{(p)}}=(\A^{d}_{\Z_{(p)}})^n$, in a diagonal way using $n$ times the above representation. 
By Lemma \ref{eke}, there exists an open subset $\cU\subset\A^{nd}_{\Z_{(p)}}$ on which $\cH\times\bG_m$ acts freely, whose image $\cV\subset\P^{nd-1}_{\Z_{(p)}}$ consists only of $\cH$-stable points, and such that the complement of $\cU$ in $\A^{nd}_{\Z_{(p)}}$ has codimension $\geq N+1$ in the fibers of $\A^{nd}_{\Z_{(p)}}\to\Spec(\Z_{(p)})$ (as $n\gg 0$). 

Let $\cZ:=\P^{nd-1}_{\Z_{(p)}}/\!\!/\cH$ be the GIT quotient (see \cite[Theorem 4]{Seshadri}). The complement of the smooth open subset $\cW:=\cV/\cH\subset \cZ$ has codimension $\geq N+1$ in the fibers of $\cZ\to\Spec(\Z_{(p)})$. Choosing a projective embedding of the projective $\Z_{(p)}$\nobreakdash-scheme~$\cZ$, one can use Poonen's Bertini theorem \cite[Theorem 1.1]{Poonen} to find an $N$\nobreakdash-dimensional complete intersection $\cY_{\F_p}$ of $\cZ_{\F_p}$ which is smooth and included in $\cW_{\F_p}$. Lifting over $\Z_{(p)}$ the equations of the projective hypersurfaces defining $\cY_{\F_p}$ in~$\cZ_{\F_p}$ yields a complete intersection $\cY$ of $\cZ$ which is smooth of relative dimension~$N$ over $\Z_{(p)}$, and included in $\cW$. 

Define $Y:=\cY_{\C}$, $V:=\cV_{\C}$ and $U:=\cU_{\C}$.
The variety $Y$ satisfies (ii) by construction. Applying Hamm's Lefschetz theorem \cite[Theorem 2 and the remark below]{Hamm} several times shows that the natural map $Y(\C)\to V(\C)/H=U(\C)/(H\times \C^*)$ is an $N$-equivalence. In the diagram of fibrations
$$U(\C)/(H\times \C^*)\leftarrow(U(\C)\times E(H\times \C^*)/(H\times \C^*)\to B(H\times \C^*),$$
the left-hand side arrow is a weak equivalence because its fibres are isomorphic to $E(H\times \C^*)$ hence contractible, and the right-hand side arrow is an $N$-equivalence as its fibres are isomorphic to $U(\C)$ hence are $(N-1)$-connected because the complement of $U(\C)$ in $\C^{nd}$ has high codimension when $n\gg 0$. 
Assertion (i) follows.
\end{proof}

When we apply Proposition \ref{approx} to the simple simply connected 
complex Lie group $E_8$ in \S\ref{partf}, we will need the following lemma.

\begin{lem}
\label{exceptional}
Fix $\ell\in\{2,3,5\}$. Then the following assertions hold.
\begin{enumerate}[(i)]
\item There exists a $16$-equivalence $f:BE_8\to K(\Z,4)$.
\item The group $H^q(BE_8,\Z)$ has no $\ell$-torsion for $q\leq 2\ell+2$.
\item There exists $y\in H^4(BE_8,\Z)$ such that the classes $\Sq^3(\rho(y))$ (if $\ell=2$) and $\beta P^1(\rho(y))$ (if~$\ell$ is odd) are non-zero.
\end{enumerate}
\end{lem}

\begin{proof}
As $\pi_1(E_8)=0$ and $E_8$ is simple, Bott \cite[Theorem~A and top of p.~253]{Bott} has shown that $\pi_2(E_8)=0$
and $\pi_3(E_8)=\Z$. 
Moreover, Bott and Samelson \cite[Theorem~V~p.~995]{BS} have proven that $\pi_q(E_8)=0$ for $4\leq q\leq 14$. 
It follows that $\pi_q(BE_8)=0$ for $1\leq q\leq 3$ and $5\leq q\leq 15$ and that $\pi_4(BE_8)=\Z$. This proves (i).

Assertion (ii) follows from (i) and from the computation of $H_*(K(\Z,4),\Z)$ given in \cite[\S 4]{Cartan}.
We now turn to (iii). In view of (i), it suffices to construct such a class $y$ in the cohomology of $K(\Z,4)$, and by the universal property of $K(\Z,4)$, it suffices to construct it in the cohomology of a space of our choice. The space $K(\Z/\ell,3)$ works because the Steenrod operations $\Sq^3\Sq^1$ (when $\ell=2$) and $\beta P^1\beta$ (when $\ell$ is odd) are non-zero in general on degree $3$ mod $\ell$ cohomology classes.
\end{proof}

In \S\ref{cexfinite}, we will use the following variant of Proposition \ref{approx}.

\begin{prop}
\label{approxfinite}
Let $p$ be a prime number.
Let $\cH$ be a finite \'etale group scheme over $\Z_{(p)}$. Fix $N\geq 1$. Then there exist a smooth projective scheme $\cY$ of relative dimension $N$ over~$\Z_{(p)}$ and a line bundle $\cL$ over $\cY$ with the following properties.
\begin{enumerate}[(i)]
\item There exists an $N$-equivalence $\cY(\C)\to B(\cH(\C)\times\C^*)$.
\item The topological complex line bundle $\cL(\C)\to \cY(\C)$ is classified by the map $\cY(\C)\to B\C^*$ appearing in (i). 
\item If $N\geq 2$ and $\ell\neq p$ is a prime number, there is a $\Gal(\oF/\F_p)$-equivariant map
\begin{equation}
\label{GaloisH1}
\Hom(\cH(\oF),\Z/\ell)=H^1(\cH(\oF),\Z/\ell)\isoto H^1_{\et}(\cY_{\oF},\Z/\ell).
\end{equation}
\end{enumerate}
\end{prop}

\begin{proof}
Construct $\cY$ exactly as in the proof of Proposition \ref{approx}.
 (As the reference \cite{Seshadri} assumes that $\cH$ has connected geometric fibers, one has to replace it with \cite[Th\'eor\`eme~1~iv)]{Raynaudquot}.) Assertion (i) is verified there.

By construction, there exists an $\cH$-torsor $\wcY\to\cY$ whose total space is a relative complete intersection in $\P^{m}_{\Z_{(p)}}$ for some $m$. The tautological bundle $\cO_{\wcY}(1)$ is naturally $\cH$-linearized, and hence descends to a line bundle $\cL$ on $\cY$. Inspecting the proof of Proposition \ref{approx} shows that (ii) holds.

As $\wcY_{\oF}$ is simply connected by \cite[XII, Corollaire~3.5]{SGA2}, the Hochschild--Serre spectral sequence for $\wcY_{\oF}\to\cY_{\oF}$ yields the isomorphism (\ref{GaloisH1}). It is canonical, hence $\Gal(\oF/\F_p)$-equivariant.
\end{proof}

\subsection{A\hspace{-.05em} torsion-free\hspace{-.05em} counterexample\hspace{-.05em} to\hspace{-.05em} the\hspace{-.05em} integral\hspace{-.05em} Tate\hspace{-.05em} conjecture\hspace{-.05em} \texorpdfstring{over~$\oF$}{}}
\label{partf}

  The proof of Proposition \ref{proptf} below is to be compared with \cite{Schoen}. In this article, Schoen overcomes the difficulty of constructing hyperplane sections of smooth projective varieties over $\oF$ whose vanishing cohomology contains interesting Tate classes. In contrast, the main ingredient of the proof of Proposition \ref{proptf} is the construction a hyperplane section whose vanishing cohomology contains no non-zero Tate classes.
  
\begin{prop}
\label{proptf}
Choose $\ell\in\{2,3,5\}$ and let $p$ be a prime number distinct from~$\ell$.
There exist a smooth projective variety $X$ of dimension $2\ell+2$ over $\oF$ and a class $x\in H^4_{\et}(X,\Z_{\ell}(2))$ such that the following assertions hold.
\begin{enumerate}[(i)]
\item The ring $H^*_{\et}(X,\Z_{\ell})$ is torsion-free.
\item A multiple of the class $x$ is algebraic.
\item Define $\alpha:=\Sq^2$ if $\ell=2$ and $\alpha:= P^1$ if $\ell$ is odd.
Then the class $\alpha(\rho(x))$
is not the reduction modulo $\ell$ of an integral Tate class.
\end{enumerate}
\end{prop}

\begin{proof} 
We split the proof in several steps.

\begin{Step}
\label{Stepapprox}
An algebraic approximation of the classifying space of $E_8$.
\end{Step}

Let $Y$ be a variety obtained by applying Proposition \ref{approx} with $H=E_8$ and $N=2\ell+3$.
Let $Z$ be a smooth projective variety over a finite field $\F$ obtained by reducing~$Y$ modulo~$p$. 
After maybe enlarging~$\F$, we may assume that $Z(\F)\neq\varnothing$. Define $\nu:W\to Z$ to be the identity if $\ell\in\{2,5\}$ and the blow-up of an $\F$-point of~$Z$ if $\ell=3$.
Set $\oZ:=Z_{\oF}$ and $\oW:=W_{\oF}$.

Proposition \ref{approx} (i), Artin's comparison theorem \cite[XI, Th\'eor\`eme~4.4]{SGA43}, the invariance of \'etale cohomology under extensions of algebraically closed fields
\cite[XVI, Corollaire~1.6]{SGA43} and the smooth and proper base change theorems 
\cite[XVI, Corollaire~2.2]{SGA43}
yield a commutative diagram
\begin{equation}
\label{red2}
\begin{aligned}
\xymatrix
@R=0.3cm
@C=0.4cm
{
H^*(BE_8,\Z_{\ell})[t]=H^*(B(E_8\times\C^*),\Z_{\ell})\ar^{}[r]\ar^{\rho}[d]& H^*_{\et}(\oZ,\Z_{\ell})\ar^{\rho}[d]\ar^{\nu^*}[r]&H^*_{\et}(\oW,\Z_{\ell})\ar^{\rho}[d] \\
H^*(BE_8,\Z/\ell)[t]=H^*(B(E_8\times\C^*),\Z/\ell)\ar^{}[r]& H^*_{\et}(\oZ,\Z/\ell)\ar^{\nu^*}[r]&H^*_{\et}(\oW,\Z/\ell)
}
\end{aligned}
\end{equation}
whose left horizontal arrows are isomorphisms in degree $\leq 2\ell+2$ and injective in degree~$2\ell+3$, whose right horizontal arrows are injective by computation of the cohomology of a blow-up, whose bottom row is $\cA_{\ell}$-equivariant, and where $t\in H^2(B\C^*,\Z)$ is a generator of $H^*(B\C^*,\Z)=H^*(\P^{\infty}(\C),\Z)=\Z[t]$. 

Since $H^{2\ell+2}_{\et}(\oZ,\Q_{\ell})=\Q_{\ell}^{\oplus s}$ with $s=2$ if $\ell=2$, $s=3$ if $\ell=3$ and $s=4$ if $\ell=5$, by~(\ref{red2}) and Lemma~\ref{exceptional}, we deduce from the computation of the cohomology of a blow-up that $H^{2\ell+2}_{\et}(\oW,\Q_{\ell})$ always has even rank.

\begin{Step}
\label{Lefschetz}
A Lefschetz pencil.
\end{Step}

Let~$\cL$ be a very ample line bundle on $W$ and fix $m\geq 3$. After maybe enlarging the finite field~$\F$, we may choose a Lefschetz pencil $f:\wW\to\P^1_{\F}$ on $W$ associated with $\cL^{\otimes 2m}$ (see \cite[XVII, Th\'eor\`eme 2.5]{SGA72}).
Let $U\subset \P^1_{\F}$ be the smooth locus of the pencil, and let $\oeta$ be a geometric generic point of $\P^1_{\F}$.

By the hard Lefschetz theorem \cite[Corollaire 4.3.9]{Weil2}, one has a decomposition
\begin{equation}
\label{hard}
H^{2\ell+2}_{\et}(\wW_{\oeta},\Q_{\ell}(\ell+1))=H^{2\ell+2}_{\et}(\oW,\Q_{\ell}(\ell+1))\oplus H^{2\ell+2}_{\et}(\wW_{\oeta},\Q_{\ell}(\ell+1))_{\van}
\end{equation}
which is orthogonal with respect to the cup-product pairing.
Let $O$ be the orthogonal group of the vanishing cohomology $H^{2\ell+2}_{\et}(\wW_{\oeta},\Q_{\ell}(\ell+1))_{\van}$ endowed with its cup-product pairing, and let $O^+\subset O$ be the special orthogonal subgroup. We view~$O$ and $O^+$ as $\ell$-adic Lie groups.
By \cite[Proposition~1.1~(iii)]{Schoen}, the image of the geometric monodromy representation
$\pi_1(U_{\oF},\oeta)\to O$ 
is infinite (in \emph{loc.\,cit.}, the proof is written for a Lefschetz pencil on a threefold, but the argument works as well in our situation). 
It follows from \cite[Th\'eor\`eme 4.4.1]{Weil2}, that this image is open for the $\ell$-adic topology. So is, a fortiori, the image $I\subset O$ of the arithmetic monodromy representation $\pi_1(U,\oeta)\to O$. We deduce that $I$ contains a non-empty open subset $\Omega$ of $O^+$.

\begin{Step}
\label{Cebotarev}
Applying the \v{C}ebotarev density theorem.
\end{Step}

By \cite[proof of Lemma 9.2.1 (i)]{Schoen2}, the dimension of $H^{2\ell+2}_{\et}(\wW_{\oeta},\Q_{\ell}(\ell+1))$ is even.
As we had checked at the end of Step \ref{Stepapprox} that $H^{2\ell+2}_{\et}(\oW,\Q_{\ell}(\ell+1))$ has even rank, it follows from (\ref{hard}) that the dimension of $H^{2\ell+2}_{\et}(\wW_{\oeta},\Q_{\ell}(\ell+1))_{\van}$ is also even, say, equal to $2d$. Let $n_0\in \bN$ be such that $\varphi(n)>2d$ for all $n\geq n_0$, where~$\varphi$ is Euler's totient function.
As there exists an element of $\SO_{2d}(\overline{\Q}_{\ell})\simeq \SO_{2d}(\C)$ none of whose eigenvalues are roots of unity, a nonempty Zariski-open subset of $\SO_{2d}(\overline{\Q}_{\ell})$ consists of matrices with no roots of unity of order $<n_0$ as eigenvalues. It follows that there exists a nonempty open subset $\Omega'\subset \Omega$ for the $\ell$-adic topology containing only matrices none of whose eigenvalues are roots of unity of order $<n_0$.

\v{C}ebotarev's theorem now
shows the existence of a closed point ${u\in U}$ such that the image in $O$ of the Frobenius $F_u$ at $u$ belongs to $\Omega'$. Denoting by $X:=(\wW_{u})_{\oF}$ the geometric fiber of $f$ above $u$, this implies that $F_u$ acts on 
the vanishing cohomology $H^{2\ell+2}_{\et}(X,\Q_{\ell}(\ell+1))_{\van}$ with no roots of unity of order $<n_0$ as eigenvalues.
By \cite[Th\'eor\`eme 1.6]{Weil1}, the characteristic polynomials of $F_u$ acting on $H^{2\ell+2}_{\et}(\oW,\Q_{\ell}(\ell+1))$ and $H^{2\ell+2}_{\et}(X,\Q_{\ell}(\ell+1))$ have rational coefficients. By hard Lefschetz \cite[Corollaire 4.3.9]{Weil2}, so is the characteristic polynomial of $F_u$ acting on $H^{2\ell+2}_{\et}(X,\Q_{\ell}(\ell+1))_{\van}$. By our choice of~$n_0$, this characteristic polynomial, which has degree $2d$, cannot annihilate any root of unity of order $\geq n_0$. Consequently, none of the eigenvalues of $F_u$ acting on $H^{2\ell+2}_{\et}(X,\Q_{\ell}(\ell+1))_{\van}$ are roots of unity. In other words, $H^{2\ell+2}_{\et}(X,\Q_{\ell}(\ell+1))_{\van}$ contains no non-zero Tate classes.

\begin{Step}
The cohomology ring $H^*_{\et}(X,\Z_{\ell})$ is torsion-free.
\end{Step}

 By the weak Lefschetz theorem, the restriction map $H^q_{\et}(\oW,\Z_{\ell})\to H^q_{\et}(X,\Z_{\ell})$ is an isomorphism for $q\leq 2\ell+1$ and is injective with torsion-free cokernel for $q=2\ell+2$ (see \cite[(4.1.6)]{Weil2}). In view of (\ref{red2}) and Lemma~\ref{exceptional}, 
this implies that $H^q_{\et}(X,\Z_{\ell})$  is torsion-free for $q\leq 2\ell+2$.
Let us show by descending induction on $q$ that $H^q_{\et}(X,\Z_{\ell})$ is also torsion-free for $q\geq 2\ell+3$. The assertion holds for $q\geq 4\ell+5$ by cohomological dimension \cite[X, Corollaire 4.3]{SGA43}. Fix $2\ell+3\leq q\leq 4\ell+4$. 
Since both $H^{q+1}_{\et}(X,\Z_{\ell})$ and $H^{4\ell+5-q}_{\et}(X,\Z_{\ell})$ are torsion\nobreakdash-free, the long exact sequences of cohomology associated with $0\to\Z_\ell\xrightarrow{\ell^m}\Z_{\ell}\to\Z/\ell^m\to 0$
 yield isomorphisms
$$H^q_{\et}(X,\Z_{\ell})/\ell^m\isoto H^q_{\et}(X,\Z/\ell^m)\textrm{ \,and }H^{4\ell+4-q}_{\et}(X,\Z_{\ell})/\ell^m\isoto H^{4\ell+4-q}_{\et}(X,\Z/\ell^m)$$
for all $m\in\bN$. As $H^q_{\et}(X,\Z/\ell^m)$ and $H^{4\ell+4-q}_{\et}(X,\Z/\ell^m)$ have the same cardinality by Poincar\'e duality (see \cite[XVIII, (3.2.6.2)]{SGA43}), and as the cardinality of $H^{4\ell+4-q}_{\et}(X,\Z_{\ell})/\ell^m$ is a linear function of $\ell^m$ because $H^{4\ell+4-q}_{\et}(X,\Z_{\ell})$ is torsion-free, we deduce that the cardinality of $H^q_{\et}(X,\Z_{\ell})/\ell^m$ is a linear function of $\ell^m$. It follows that $H^q_{\et}(X,\Z_{\ell})$ is torsion-free. We have thus proven assertion (i).

\begin{Step}
The cohomology class $x$.
\end{Step}

Let $y\in H^4(BE_8,\Z)$ be the class given by Lemma~\ref{exceptional}, let $z\in H^4_{\et}(\oW,\Z_{\ell})$ be the class it induces via (\ref{red2}), and let $x\in H^4_{\et}(X,\Z_{\ell}(2))$ be the class obtained from $z|_X\in H^4_{\et}(X,\Z_{\ell})$ by choosing an isomorphism $\Z_{\ell}\isoto\Z_{\ell}(2)$ over $\oF$.

By Edidin and Graham \cite[Theorem 1(c)]{EG} (see also \cite[Theorem 2.14]{Totarobook}), the cycle class map $\CH^2(B(E_8\times\C^*))\otimes_{\Z}\Q\to H^4(B(E_8\times\C^*),\Q)$ is surjective. Consequently, a multiple of the class $y$, hence also of $x$, is algebraic, proving~(ii).

Assume by contradiction that $\alpha(\rho(x))\in H^{2\ell+2}_{\et}(X,\bmu_{\ell}^{\otimes 2})=H^{2\ell+2}_{\et}(X,\bmu_{\ell}^{\otimes \ell +1})$ is the reduction modulo $\ell$ of an integral Tate class $w\in H^{2\ell+2}_{\et}(X,\Z_\ell(\ell+1))$. By the weak Lefschetz theorem (see \cite[(4.1.6)]{Weil2}), there is an exact sequence
\begin{equation}
\label{ses}
0\to H^{2\ell+2}_{\et}(\oW,\Z_{\ell}(\ell+1))\to H^{2\ell+2}_{\et}(X,\Z_{\ell}(\ell+1))\to K\to 0
\end{equation}
with $K$ torsion-free. Since $K\otimes_{\Z_{\ell}}\Q_{\ell}=H^{2\ell+2}_{\et}(X,\Q_{\ell}(\ell+1))_{\van}$ contains no non-zero Tate class by Step \ref{Cebotarev}, the image of $w$ in $K$ is torsion, hence zero because $K$ is torsion-free. By (\ref{ses}), we deduce the existence of a Tate class $v\in H^{2\ell+2}_{\et}(\oW,\Z_\ell(\ell+1))$ with $\alpha(\rho(x))=\rho(v|_X)$.
As the restriction map $H^{2\ell+2}_{\et}(\oW,\bmu_{\ell}^{\otimes \ell +1})\to H^{2\ell+2}_{\et}(X,\bmu_{\ell}^{\otimes \ell +1})$ is injective, again by weak Lefschetz \cite[(4.1.6)]{Weil2}, one has $\alpha(\rho(z))=\rho(v)$ in $H^{2\ell+2}_{\et}(\oW,\bmu_{\ell}^{\otimes \ell +1})$. Applying $\beta$ to this identity shows that $\beta\alpha(\rho(z))=0$, because~$\beta$ kills the reductions mod $\ell$ of integral classes.
This is a contradiction because $\beta\alpha(\rho(z))$ is non-zero by~(\ref{red2}) and Lemma~\ref{exceptional}. We have finally proved~(iii).
\end{proof}

\begin{thm}
\label{th2b}
Choose $\ell\in\{2,3,5\}$ and let $p$ be a prime number distinct from~$\ell$.
There exists a smooth projective variety $X$ of dimension $2\ell+2$ over $\oF$ with $H^*_{\et}(X,\Z_{\ell})$ torsion-free on which the $\ell$-adic integral Tate conjecture for codimension~$2$ cycles fails.
\end{thm}

\begin{proof}
Let $X$ and $x$ be as in Proposition \ref{proptf}. The class $x$ is Tate because a multiple of it is algebraic by Proposition \ref{proptf} (ii). It cannot be algebraic itself by Corollary~\ref{Tateobstr}~(ii) and Proposition~\ref{proptf} (iii).
\end{proof}

\begin{rem}
\label{rem2}
Fix $\ell\in\{2,3,5\}$.
The proofs of Proposition~\ref{proptf} and Theorem \ref{th2b} may be adapted to show that, over the complex numbers, a $(2\ell+2)$-dimensional complete intersection of very general sufficiently ample hypersurfaces in a smooth projective variety obtained by applying Proposition~\ref{approx} to $H=E_8$
fails the integral Hodge conjecture for codimension $2$ cycles (but their cohomology ring has no $\ell$-torsion, and in fact no torsion at all when $\ell=2$). 

The main modifications to implement are the following. First, one has to use Corollary~\ref{Hodgeobstr}~(ii) instead of Corollary~\ref{Tateobstr}~(ii).
More importantly, instead of verifying that the vanishing cohomology contains no non-zero Tate classes as in Steps~\ref{Lefschetz} and~\ref{Cebotarev} of the proof of Proposition \ref{proptf}, one needs to verify that it contains no non-zero Hodge classes. This is a consequence of a Noether--Lefschetz theorem (see \cite[Th\'eor\`eme 18.28]{Voisinbook}). 
 \end{rem}

\subsection{Geometrically trivial non-algebraic classes over all finite fields}
\label{cexfinite}

In this paragraph, we work over a finite field $\F$ of characteristic $p$, with algebraic closure~$\oF$. If $X$ is a variety over~$\F$, we set $\oX:=X_{\oF}$.
 Let  $G:=\Gal(\oF/\F)$ be the absolute Galois group of $\F$, and $F\in G$ be the geometric Frobenius (see \cite[(1.15.1)]{Weil1}). 
Recall that $G$ has cohomological dimension $1$ and that, if $M$ is a finite $G$-module, one has a natural isomorphism 
\begin{equation}
\label{H1}
M/(F-\Id)\isoto H^1(G,M).
\end{equation}

\begin{prop}
\label{constrFq}
Let $p\neq \ell$ be prime numbers, and let $\F$ be a finite field of characteristic $p$.
There exist a smooth projective geometrically connected  variety~$X$ of dimension $2\ell+3$ over~$\F$ and a class $z\in H_{\et}^2(\oX,\bmu_{\ell}^{\otimes 2})$ such that the class  $\Sq^3\Sq^1(z)$ (if $\ell=2$), or $\beta P^1\beta(z)$ (if $\ell$~is odd), does not belong to the image of $F-\Id$.
\end{prop}

\begin{proof}
Consider the \'etale abelian sheaf $(\bmu_{\ell})^2$ 
over $\Z_{(p)}$ and view it as a finite \'etale group scheme~$\cH$
 over $\Z_{(p)}$. Let $\cY$ be a smooth projective scheme over~$\Z_{(p)}$ obtained by applying Proposition \ref{approxfinite} to $\cH$, with $N=2\ell+3$. Define $X:=\cY_{\F}$.

Proposition \ref{approxfinite} (i), Artin's comparison theorem \cite[XI, Th\'eor\`eme~4.4]{SGA43}, the invariance of \'etale cohomology under extensions of algebraically closed fields
\cite[XVI, Corollaire~1.6]{SGA43} and the smooth and proper base change theorems 
\cite[XVI, Corollaire~2.2]{SGA43} yield a morphism
\begin{equation*}
\Phi:H^*(B((\Z/\ell)^2\times\C^*),\Z/\ell)\to H^*_{\et}(\oX,\Z/\ell)
\end{equation*}
which is an isomorphism in degree $\leq 2\ell+2$, and which is $\cA_{\ell}$-equivariant. That~$\Phi$ is an isomorphism in degree $0$ implies that $X$ is geometrically connected.

It is well-known that $H^*(B\C^*,\Z/\ell)=\Z/\ell[t]$ with $t\in H^2(B\C^*,\Z/\ell)$, that
$H^*(B\Z/2,\Z/2)=\Z/2[x]$ with $x\in H^1(B\Z/2,\Z/2)$, and that, when $\ell$ is odd, $H^*(B\Z/\ell,\Z/\ell)=\Z/\ell[x,y]/\langle x^2\rangle$ with $x\in H^1(B\Z/\ell,\Z/\ell)$ and $y=\beta(x)$.
It then follows from the K\"unneth formula that
$$H^*(B((\Z/2)^2\times\C^*),\Z/2)=\Z/2[x_1,x_2,t],$$
and that, when $\ell$ is odd, one has 
$$H^*(B((\Z/\ell)^2\times\C^*),\Z/\ell)=\Z/\ell[x_1,y_1,x_2,y_2,t]/\langle x_1^2,x_2^2\rangle.$$

We now compute Galois actions. Let $|\F|=p^a$ be the cardinality of $\F$. By $G$-equivariance of (\ref{GaloisH1}) and our choice of~$\cH$, one has $F(\Phi(x_i))=p^a\Phi(x_i)$ in $H^1_{\et}(\oX,\Z/\ell)$ for $i\in\{1,2\}$. 
When $\ell$ is odd, one also has $F(\Phi(y_i))=p^a\Phi(y_i)$ in $H^2_{\et}(\oX,\Z/\ell)$ for $i\in\{1,2\}$ by $G$-equivariance of the Bockstein. In addition, since $\Phi(t)\in H^2_{\et}(\oX,\Z/\ell)$ is the first Chern class of a line bundle which is defined over $\F$, by Proposition \ref{approxfinite} (ii), one has $F(\Phi(t))=p^a\Phi(t)$ in $H^2_{\et}(\oX,\Z/\ell)$.

Define $w:=\Phi(x_1x_2)$.  When $\ell=2$, one has 
$$\Sq^3\Sq^1(w)=\Phi(\Sq^3\Sq^1(x_1x_2))=\Phi(\Sq^3(x_1^2x_2+x_1x_2^2))=\Phi(x_1^4x_2^2+x_1^2x_2^4).$$
When $\ell$ is odd, one computes similarly that
\begin{alignat*}{5}
\beta P^1\beta(w)&=\Phi(\beta P^1\beta(x_1x_2))=\Phi(\beta P^1(y_1x_2-x_1y_2))\\
&=\Phi(\beta(y^{\ell}_1x_2-x_1y^{\ell}_2))=\Phi(y_1^{\ell}y_2-y_1y_2^{\ell}).
\end{alignat*}
Choose a class $\xi\in H^0_{\et}(\oX,\bmu_{\ell}^{\otimes 2})$ and define $z:=\xi\cdot w\in H_{\et}^{2}(\oX,\bmu_{\ell}^{\otimes 2})$.
The action of~$F$ on the group $H_{\et}^{2\ell+2}(\oX,\bmu_{\ell}^{\otimes 2})$ is diagonal in a basis formed by multiples by $\xi$ of appropriate monomials in $\Phi(t)$, $\Phi(x_1)$, $\Phi(x_2)$ (and, when $\ell$ is odd, $\Phi(y_1)$ and~$\Phi(y_2)$).
Paying attention to the Tate twist, we see that the elements $\xi\cdot\Phi(x_1^4x_2^2)$ (when $\ell=2$) and 
$\xi\cdot\Phi(y^{\ell}_1y_2)$ (when~$\ell$ is odd) of this group are $F$-stable.
 Taking Remark~\ref{rems}~(iv) into account, we deduce that $\Sq^3\Sq^1(z)=\xi\cdot\Phi(x_1^4x_2^2+x_1^2x_2^4)$ (when $\ell=2$) and $\beta P^1\beta(z)=\xi\cdot\Phi(y_1^{\ell}y_2-y_1y_2^{\ell})$ (when $\ell$ is odd), are not in the image of $F-\Id$.
\end{proof}

The following theorem was proven by Scavia and Suzuki \cite[Theorem 1.3]{ScaSuz} when $\F$ contains a primitive $\ell^2$-th root of unity. 

\begin{thm}
\label{finitethm}
Let $p\neq \ell$ be prime numbers, and let $\F$ be a finite field of characteristic $p$.
There exist a smooth projective geometrically connected
variety $X$ of dimension $2\ell+3$ over~$\F$ and a non-algebraic class 
$$x\in \Ker\big(H^4_{\et}(X,\Z_{\ell}(2))\to H^4_{\et}(\oX,\Z_{\ell}(2))\big).$$
\end{thm}

\begin{proof}
Let $X$ and~$z$ be as in Proposition \ref{constrFq}. The inverse limit, when $m$ varies, of the short exact sequences 
\begin{equation}
\label{HS}
0\to H^1(G,H^{q-1}_{\et}(\oX,\bmu_{\ell^m}^{\otimes r}))\to H^q_{\et}(X,\bmu_{\ell^m}^{\otimes r})\to H^q_{\et}(\oX,\bmu_{\ell^m}^{\otimes r})^{G}\to 0
\end{equation}
stemming from Hochschild--Serre spectral sequences yield short exact sequences
\begin{equation}
\label{HS2}
0\to H^1(G,H^{q-1}_{\et}(\oX,\Z_{\ell}(r)))\to H^q_{\et}(X,\Z_{\ell}(r))\to H^q_{\et}(\oX,\Z_{\ell}(r))^{G}\to 0
\end{equation}
because the groups $H^q_{\et}(\oX,\bmu_{\ell^m}^{\otimes r}))$ are finite and the $G$-cohomology groups of a finite $G$-module are finite.

Denote by $\psi_X: H^{q-1}_{\et}(\oX,\bmu_{\ell^m}^{\otimes r})\to H^q_{\et}(X,\bmu_{\ell^m}^{\otimes r})$ the compositions
$$H^{q-1}_{\et}(\oX,\bmu_{\ell^m}^{\otimes r})\to H^{q-1}_{\et}(\oX,\bmu_{\ell^m}^{\otimes r})/(F-\Id)\to H^1(G,H^{q-1}_{\et}(\oX,\bmu_{\ell^m}^{\otimes r}))\to H^q_{\et}(X,\bmu_{\ell^m}^{\otimes r})$$
of the quotient map, of the isomorphism (\ref{H1}), and of the left arrow of (\ref{HS}). Consider the element $y:=\psi_X(\beta(z))$ of $ H^4_{\et}(X,\bmu_{\ell}^{\otimes 2})$. As $\beta$ is the Bockstein, 
$\beta(z)$ lifts compatibly to $H^3_{\et}(\oX,\bmu_{\ell^m}^{\otimes 2})$ for all $m$. Applying $\psi_X$ to these  compatible lifts, we deduce from (\ref{HS2}) that $y$ lifts to a class $x\in H^1(G,H^3_{\et}(\oX,\Z_{\ell}(2)))\subset H^4_{\et}(X,\Z_{\ell}(2))$. 
Moreover, it follows from~(\ref{HS2}) that $x\in \Ker\big(H^4_{\et}(X,\Z_{\ell}(2))\to H^4_{\et}(\oX,\Z_{\ell}(2))\big)$.

It remains to prove that $x$ is not algebraic.
 We will prove the stronger statement that $y$ is not algebraic. 
 Recall from \S\ref{Steenrodtwistedetale} the definition of the finite \'etale cover $\pi:S'_{\ell}\to S_{\ell}$  with Galois group $\Gamma:=(\Z/\ell)^*$, set $X':=X\times_{S_{\ell}}S'_{\ell}$ and 
$\oX':=\oX\times_{S_{\ell}}S'_{\ell}$, and define $\psi_{X'}$ in the same way as $\psi_X$.
For $\alpha=\beta P^1$ (if $\ell$ is odd), or $\alpha\in\{\Sq^2,\Sq^3\}$ (if $\ell=2$), the diagram
\begin{equation*}
\begin{aligned}
\xymatrix
@R=0.3cm
{
 H^{3}_{\et}(\oX',\Z/\ell)\ar^{\alpha\hspace{1.5em}}[r]\ar^{\psi_{X'}}[d]&H^{3+\deg(\alpha)}_{\et}(\oX',\Z/\ell)\ar^{\psi_{X'}}[d]  \\
H^4_{\et}(X',\Z/\ell)\ar^{\alpha\hspace{1.5em}}[r]&H^{4+\deg(\alpha)}_{\et}(X',\Z/\ell)
}
\end{aligned}
\end{equation*}
commutes by \cite[Corollary 3.5]{ScaSuz}. Retaining the appropriate isotypic components for the action of $\Gamma$, we deduce the commutativity of the following diagram:
\begin{equation}
\label{SteenHS}
\begin{aligned}
\xymatrix
@R=0.3cm
{
 H^{3}_{\et}(\oX,\bmu_{\ell}^{\otimes r})\ar^{\alpha\hspace{1.5em}}[r]\ar^{\psi_X}[d]&H^{3+\deg(\alpha)}_{\et}(\oX,\bmu_{\ell}^{\otimes r})\ar^{\psi_X}[d]  \\
H^4_{\et}(X,\bmu_{\ell}^{\otimes r})\ar^{\alpha\hspace{1.5em}}[r]&H^{4+\deg(\alpha)}_{\et}(X,\bmu_{\ell}^{\otimes r})\rlap{.}
}
\end{aligned}
\end{equation}

Assume first that $\ell$ is odd. As $\beta P^1\beta(z)$ does not  belong to the image of ${F-\Id}$ by Proposition \ref{constrFq}, one has $\psi_X(\beta P^1\beta(z))\neq 0$, hence $\beta P^1(y)\neq 0$ by the commutativity of (\ref{SteenHS}) for $\alpha=\beta P^1$. It follows from Proposition~\ref{stability}~(ii) and Lemma \ref{compaBock} that $\beta P^1$ kills $H^4_{\et}(X,\bmu_{\ell}^{\otimes 2})_{\alg}$. Hence~$y$ cannot be algebraic.

Assume now that $\ell=2$. As $\Sq^3\Sq^1(z)$ does not  belong to the image of ${F-\Id}$ by Proposition \ref{constrFq}, one has $\psi_X(\Sq^3\Sq^1(z))\neq 0$, hence $\Sq^3(y)\neq 0$ by the commutativity of (\ref{SteenHS}) for $\alpha=\Sq^3$. The commutativity of (\ref{SteenHS}) for $\alpha=\Sq^2$ shows that $\Sq^2(y)\in H^1(G,H^{5}_{\et}(\oX,\bmu_{2}^{\otimes 2}))\subset H^6_{\et}(X,\bmu_{2}^{\otimes 2})$. In view of (\ref{HS}), one thus has
$$\Sq^2(y)\in F^1H^6_{\et}(X,\bmu_{2}^{\otimes 2})=\Ker\big(H^6_{\et}(X,\bmu_{2}^{\otimes 2})\to H^6_{\et}(\oX,\bmu_{2}^{\otimes 2})\big),$$
where $F^{\bullet}H^q_{\et}(X,\bmu_{2}^{\otimes r})$ denotes the abutment filtration of the Hochschild--Serre spectral sequence.
Since $\varpi\in F^1H^1_{\et}(X,\Z/2)=\Ker\big(H^1_{\et}(X,\Z/2)\to H^1_{\et}(\oX,\Z/2)\big)$, the multiplicativity properties of Hochschild--Serre spectral sequences imply that 
$\varpi\cdot\Sq^2(y)\in F^2H^7_{\et}(X,\bmu_{2}^{\otimes 2})$. This group vanishes because $G$ has cohomological dimension $1$, so that $\varpi\cdot\Sq^2(y)=0$. It follows that $(\Sq^3+\varpi\cdot \Sq^2)(y)\neq 0$. The class~$y$ is not algebraic because
$\Sq^3+\varpi\cdot\Sq^2$ kills $H^4_{\et}(X,\bmu_{2}^{\otimes 2})_{\alg}$ 
by Proposition~\ref{weirdobstr}~(ii).
\end{proof}

\subsection{A\hspace{-.05em} fourfold\hspace{-.05em} failing\hspace{-.05em} the\hspace{-.05em} integral\hspace{-.05em} Hodge\hspace{-.05em} conjecture\hspace{-.05em} over\hspace{-.05em} \texorpdfstring{$\R$}{R} but\hspace{-.05em} not\hspace{-.05em} over\hspace{-.05em}~\texorpdfstring{$\C$}{C}}
\label{cexreal}

In this  paragraph, we work over $\R$. Let $G:=\Gal(\C/\R)\simeq \Z/2$ be generated by the complex conjugation $\sigma$. Let $\Z(k)$ be the $G$-module whose underlying abelian group is $\Z$, on which $\sigma$ acts by~$(-1)^k$.
Let ${\omega\in H^1(G,\Z(1))=H^1_G(\pt,\Z(1))\simeq \Z/2}$ and $\varpi\in H^1(G,\Z/2)=H^1_G(\pt,\Z/2)\simeq \Z/2$ and  be the generators.  One has $\rho(\omega)=\varpi$. We still denote by $\omega\in H^1_G(X(\C),\Z(1))$ and $\varpi\in H^1_G(X(\C),\Z/2)$ their pull-backs to the $G$-equivariant Betti cohomology of any variety $X$ over $\R$.  The class $\varpi\in H^1_G(X(\C),\Z/2)$ corresponds to the one defined in \S\ref{parBockstein} through the isomorphism $H^*_{\et}(X,\Z/2)\isoto H^*_G(X(\C),\Z/2)$ between \'etale cohomology and equivariant Betti cohomology \cite[Corollary 15.3.1]{Scheidererbook}. If $X$ is smooth, Krasnov (\cite[\S 2.1]{Kraccvb}, see also \cite[\S 1.6.1]{BW1}) has defined, for all $c\geq 0$, a cycle class map
\begin{equation}
\label{eqccm}
\cl:\CH^c(X)\to H^{2c}_G(X(\C),\Z(c)).
\end{equation}
The classes in the image of (\ref{eqccm}) are said to be algebraic.

Let $E$ be the real elliptic curve with set of complex points $E(\C)=\C/(\Z\oplus i\Z)$, on which $\sigma\in G$ acts through the usual complex conjugation of $\C$.

\begin{prop}
\label{realce}
There exist $a\in H^1_G(E(\C),\Z)$ and $b\in H^1_G(E(\C),\Z(1))$ such that $a\cdot b\in H^2_G(E(\C),\Z(1))$ is the cycle class of a real point $P\in E(\R)$. 
\end{prop}

\begin{proof}
One has $H^1(E(\C),\Z)\simeq \Z\oplus \Z(1)$ as $G$-modules, where the first factor is generated by the class $a'\in H^1(E(\C),\Z)$ of the loop $\R/\Z\hookrightarrow \C/(\Z \oplus i\Z)$ and the second factor by the class $b'\in H^1(E(\C),\Z)$ of the loop $i\R/i\Z\hookrightarrow \C/(\Z \oplus i\Z)$. We note that $a'\cdot b'\in H^2(E(\C),\Z)$ has degree $1$. 

By work of Krasnov, the Hochschild--Serre spectral sequences
$$E_2^{s,t}=H^s(G,H^t(E(\C),\Z(k))\implies H^{s+t}_G(E(\C),\Z(k))$$
degenerate (combine \cite[\S 5.1, Proposition 3.6, Corollary 3.4]{KraHarnack} and \cite[Remark 1.2]{KraGMZ}). 
It follows that there exist $a\in H^1_G(E(\C),\Z)$ and $b\in H^1_G(E(\C),\Z(1))$ lifting $a'$ and $b'$. By Krasnov's real Lefschetz $(1,1)$ theorem (see for instance \mbox{\cite[IV, Theorem 4.1]{vanHamelthesis}}), the class $a\cdot b\in H^2_G(E(\C),\Z(1))$ is the cycle class of a divisor $D$ on $E$. As $a'\cdot b'$ has degree $1$, the divisor $D$ has degree $1$. The Riemann--Roch theorem applied on the elliptic curve $E$ now shows that $D$ is effective, hence rationally equivalent to a point $P\in E(\R)$.
\end{proof}

Let $Q:=\{\sum_{i=0}^{4} X_i^2=0\}\subset\P^{4}_{\R}$ be the $3$-dimensional anisotropic quadric.  
We consider the smooth projective real algebraic fourfold $X:=E\times_{\R} Q^3$. 

\begin{prop}
\label{realtorsionclass}
There is a class $x\in \Ker\big(H^4_G(X(\C),\Z(2))\to H^4(X(\C),\Z(2))\big)$ such that $\Sq^3(\rho(x))+\varpi\cdot\Sq^2(\rho(x))+\varpi^3\cdot\rho(x)\neq 0$ in $H^7_G(X(\C),\Z/2)$.
\end{prop}

\begin{proof}
Let $a$, $b$ and $P$ be as in Proposition \ref{realce}.
 Let $f:X\to E$ and $g:X\to Q$ be the projections. Set $x:=f^*b\cdot\omega^3\in H^4_G(X(\C),\Z(4))$. View it as an element of $H^4_G(X(\C),\Z(2))$ using an isomorphism of $G$-modules $\Z(4)\simeq \Z(2)$. 
The image of $x$ in $H^4(X(\C),\Z(2))$ vanishes because so does the image of $\omega$ in $H^1(X(\C),\Z(1))$.

As the elements $\varpi$ and $f^*\rho(b)$ of $H^1_G(X(\C),\Z/2)$ lift to $H^1_G(X(\C),\Z(1))$ and as~$\Sq^1$ is the Bockstein, one has $\Sq^1(\varpi)=\varpi^2$ and $\Sq^1(f^*\rho(b))=f^*\rho(b)\cdot\varpi$ (see for instance \cite[\S 2.2]{periodindex}).
Since $\rho(x)=f^*\rho(b)\cdot\varpi^3$ in $H^4_G(X(\C),\Z/2)$, the Cartan formula implies that
\begin{equation*}
\label{computaSq}
\Sq^3(\rho(x))+\varpi\cdot\Sq^2(\rho(x))+\varpi^3\cdot\rho(x)= 
f^*\rho(b)\cdot\varpi^6
\end{equation*}
 in $H^7_G(X(\C),\Z/2)$.
This class is non-zero because the element 
$$g_*(f^*\rho(a)\cdot f^*\rho(b)\cdot\varpi^6)=g_*(f^*\rho(\cl(P))\cdot\varpi^6)=\rho(g_*f^*\cl(P))\cdot\varpi^6=\varpi^6$$
of $H^6_G(Q(\C),\Z/2)$ is non-zero by \cite[Proposition 6.1]{Wu}.
\end{proof}

We now reach the goal of this paragraph.

\begin{thm}
\label{threal}
There exists a smooth projective variety $X$ of dimension $4$ over~$\R$ such that $X(\R)=\varnothing$,
and a non-algebraic class 
$$x\in \Ker\big(H^4_G(X(\C),\Z(2))\to H^4(X(\C),\Z(2))\big).$$
\end{thm}

\begin{proof}
Let $X$ be as above and let $x\in H^4_G(X(\C),\Z(2))$ be as in Proposition \ref{realtorsionclass}. One has $X(\R)=\varnothing$ because $Q^3$ is anisotropic. The inverse image $y\in H^4_{\et}(X,\bmu_2^{\otimes 2})$ of~$\rho(x)$ by the comparison isomorphism $H^4_{\et}(X,\bmu_2^{\otimes 2})\isoto H^4_G(X(\C),\Z/2(2))$ of \cite[Corollary 15.3.1]{Scheidererbook} (which is $\cA_2$-equivariant, see \S\ref{Steenrodetale}) is not algebraic by Propositions \ref{weirdobstr}~(ii) and~\ref{realtorsionclass}. As a consequence, neither is $x$. (This last assertion follows from the compatibility of cycle class maps in \'etale and equivariant Betti cohomology. This result holds true, but we know no published reference. Alternatively, one can use the fact that the degree $2c$ classes that are algebraic are exactly those of coniveau $\geq c$, both in \'etale and equivariant Betti cohomology, by purity (for which see \cite[XIX, Th\'eor\`emes 3.2 et 3.4]{SGA43} and \cite[(1.21)]{BW1}).)
\end{proof}


If $X$ is a smooth projective variety over $\R$, the real integral Hodge conjecture for codimension~$c$ cycles on $X$, defined in \cite[Definition 2.2]{BW1}, is the statement that all the classes of $H^{2c}_G(X(\C),\Z(c))$ whose images in $H^{2c}(X(\C),\Z)$ are Hodge, and that satisfy a topological condition described in \cite[Definition 1.19]{BW1} (and which is always satisfied if $X(\R)=\varnothing$), are algebraic.

\begin{cor}
\label{corIHCreal}
There exists a smooth projective variety $X$ of dimension $4$ over~$\R$ such that $X$ fails the real integral Hodge conjecture but $X_{\C}$ satisfies the usual (complex) integral Hodge conjecture.
\end{cor}

\begin{proof}
Let $X$ and $x$ be as in Theorem \ref{threal}. The class $x\in H^4_G(X(\C),\Z(2))$ is Hodge because it vanishes in $H^4(X(\C),\Z(2))$. It satisfies the topological condition appearing in the definition of the real integral Hodge conjecture because ${X(\R)=\varnothing}$.  As $x$ is not algebraic, the variety $X$ does not satisfy the real integral Hodge conjecture.

Finally, as $H^*(Q^3(\C),\Z)$ is algebraic and torsion-free and $H^{2*}(E(\C),\Z)$ is algebraic, the K\"unneth formula implies that $H^{2*}(X(\C),\Z)$ is algebraic, hence that $X_{\C}$ satisfies the integral Hodge conjecture.
\end{proof}

\begin{rems}
\label{remIHCR}
(i)
Three-dimensional counterexamples to the real integral Hodge conjecture are constructed in \cite[Examples 4.7 and 4.8]{BW1}. This is the smallest dimension possible by \cite[Propositions 2.8 and 2.10]{BW1}.

(ii)
The counterexamples to the real integral Hodge conjecture mentioned in~(i) are explained by a failure of the complex integral Hodge conjecture. A $7$\nobreakdash-dimensional counterexample to the real integral Hodge conjecture for which the complex integral Hodge conjecture holds is described in \cite[Example 2.5]{BW1}. The fourfold of Theorem \ref{threal} is the first counterexample of this kind in dimension $\leq 6$. We do not know if some exist in dimension $3$ (a question closely related to  
\cite[Question~4.9]{BW1}).

(iii)
Over the non-archimedean real closed field $R:=\cup_n\R((t^{1/n}))$, there exist smooth projective threefolds failing the (analogue in this context of the) real integral Hodge conjecture, but for which the (analogue of the) complex integral Hodge conjecture holds, see \cite[Examples 9.5, 9.6, 9.7, 9.10 and 9.22]{BW2}.
\end{rems}

\subsection{Applications to unramified cohomology} 
\label{unramified}

Let $X$ be a smooth variety over a field $k$. Let $\ell$ be a prime number invertible in $k$, and let $$\cl:\CH^2(X)\otimes_{\Z}\Z_{\ell}\to H^4_{\cont}(X,\Z_{\ell}(2))$$
be Jannsen's cycle class map to continuous \'etale cohomology (see \cite{Jannsen}). (For the concrete fields considered in this article, such as algebraically closed fields, real closed fields, or finite fields, the continuous \'etale cohomology group 
$H^4_{\cont}(X,\Z_{\ell}(2))$ coincides with the $\ell$-adic cohomology group $H^4_{\et}(X,\Z_{\ell}(2)):=\varprojlim_m H^4_{\et}(X,\bmu_{\ell^m}^{\otimes 2})$; see \cite[Remark~3.5~c)]{Jannsen}.)

As was discovered by Colliot-Th\'el\`ene and Voisin over the complex numbers \cite[Th\'eor\`eme~3.7]{CTV}, and extended to arbitrary fields by Kahn \cite[Th\'eo\-r\`eme~1.1]{Kahnetale}, the torsion subgroup~$T$ of the cokernel of $\cl$ is a quotient of the third unramified cohomology group  $H^3_{\nr}(X,\Q/\Z(2))$. It follows from the proofs of Theorems \ref{th1b}, \ref{th2b}, \ref{finitethm} and \ref{threal} that the smooth projective varieties $X$ considered there satisfy $T\neq 0$. As a consequence, one has $H^3_{\nr}(X,\Q/\Z(2))\neq 0$ for any of these varieties.

\section{Algebraizability of cohomology classes of \texorpdfstring{$\ci$}{smooth} manifolds}

In this section,  we study the algebraizability of cohomology classes of compact~$\ci$ manifolds. We refer to \S\ref{realintro} for context and for the relevant definitions.

\subsection{Steenrod operations and algebraizable classes}

Our main results (Theorems \ref{propnonalgebraizable} and \ref{cexSWsub} below) both rely on the fact that Steenrod operations preserve algebraizable classes.
This statement follows from the next proposition, which is due to Akbulut and King \cite[Theorem~6.6]{AK} at least when $X$ is projective. 
For the convenience of the reader, we give a short proof based on the relative Wu theorem of Atiyah and Hirzebruch \cite[Satz 3.2]{AHOp}. Our argument is also closely related to \cite[Proposition 5.19]{BH}.

\begin{prop}
\label{BHSteen}
Let $X$ be a smooth variety over $\R$.  Then $H^*_{\alg}(X(\R),\Z/2)$ is stable under mod $2$ Steenrod operations.
\end{prop}

\begin{proof}
Let $g:Z\hookrightarrow X$ be an integral subvariety of codimension $c$ of $X$ and let $\nu:Y\to Z$ be a resolution of singularities of $Z$ (see \cite{Hironaka}). Set $f:=g\circ\nu$ and denote by $f(\R):Y(\R)\to X(\R)$ the map induced between real point sets. Let $w(N_{f(\R)})$ be the total Stiefel--Whitney class of the virtual normal bundle ${N_{f(\R)}:=f^*T_{X(\R)}-T_{Y(\R)}}$ of~$f(\R)$.
The relative Wu theorem of Atiyah and Hirzebruch \cite[Satz 3.2]{AHOp} applied to the map $f(\R)$ and to the class $1\in H^0(Y(\R),\Z/2)$ (exactly as in \cite[(1.60) and Remark~1.6 (i)]{BW1}) implies that 
\begin{equation}
\label{relativeWu}
\Sq(f(\R)_*1)=f(\R)_*w(N_{f(\R)}).
\end{equation}

One has $w(T_{X(\R)})\in H^*_{\alg}(X(\R),\Z/2)$ and $w(T_{Y(\R)})\in H^*_{\alg}(Y(\R),\Z/2)$ (see \cite[Theorem 1.5]{BKHomology} when $X$ and~$Y$ are projective; in general, one may combine \cite[Theorem 4]{Kahn} and \cite[Theorem 1.18]{BW1}). As algebraic classes are stable under pull-backs,  proper pushforwards and cup-product (see \cite[Theorems~1.3 and~1.4]{BKHomology} under projectivity hypotheses and \cite[\S 1.6.2]{BW1} in general), the class~(\ref{relativeWu}) is algebraic. Since $H^*_{\alg}(X(\R),\Z/2)$ is generated by classes of the form $\cl_{\R}(Z)=f(\R)_*1$, the proof is now complete.
\end{proof}

Proposition \ref{BHSteen} implies at once the following corollary.

\begin{cor}
\label{BHSteencor} 
Let $M$ be a compact $\ci$ manifold. The subset of $H^*(M,\Z/2)$ consisting of algebraizable classes is stable under mod $2$ Steenrod operations.
\end{cor}

\subsection{Examples of non-algebraizable classes}

Here is our first main result.

\begin{thm}
\label{propnonalgebraizable}
For all $c\geq 2$, there exists a compact $\ci$ manifold $M$ and a class $x\in H^c(M,\Z/2)$ that is not algebraizable.
\end{thm}

\begin{proof}
Grant and Sz\"ucs have constructed in \cite[Proof of Theorem 1.1 p.\,334]{GS} a compact $\ci$ manifold $M$ and a class $x\in H^c(M,\Z/2)$ such that $x^2$ is not in the image of the reduction modulo $2$ map $H^{2c}(M,\Z)\to  H^{2c}(M,\Z/2)$ (if $c$ is even), or such that $(\Sq^2\Sq^1 x)^2$ is not in the image of the reduction modulo $2$ map $H^{2c+6}(M,\Z)\to  H^{2c+6}(M,\Z/2)$ (if $c$ is odd).

Akbulut and King have shown in \cite[Theorem A (b)]{AKtr} (see also \cite[Remark 4.8]{Kraequivariant}) that squares of algebraizable classes are reductions modulo $2$ of integral classes. 
We deduce that $x$ is not algebraizable (if $c$ is even) and that $\Sq^2\Sq^1 x$ is not algebraizable (if $c$ is odd). Theorem \ref{BHSteencor} now shows that $x$ is not algebraizable even when $c$ is odd.
\end{proof}

\begin{rem}
When $c$ is even, the proof of Proposition \ref{propnonalgebraizable} is exactly the one given in \cite[Theorem 1.13]{Kuchomo}.
\end{rem}

\subsection{Examples of algebraizable classes}

Before stating our second main result, we recall the properties of Thom spaces that will be used in its proof.
Let $\BO(n)$ be the classifying space for $O(n)$, constructed as the increasing union of the Grassmannians $\Grass(n,\R^{n+m})$ of $n$-planes in $\R^{n+m}$ (for $m\geq 0$). Let $E(n)$ be the tautological rank $n$ topological real vector bundle on $\BO(n)$ ($\BO(n)$ and $E(n)$ are denoted by $G_n$ and $\gamma^n$ in \cite[p.\,63]{MS}). One has 
\begin{equation}
\label{cohograss}
H^*(\BO(n),\Z/2)=\Z/2[w_1,\dots,w_n],
\end{equation}
 where the $w_i$ are the Stiefel--Whitney classes of $E(n)$ (see \cite[Theorem 7.1]{MS}). 

Let $\MO(n)$ be the Thom space of $E(n)$ (defined in \cite[p.\,205]{MS}). Denoting by $s_n\in H^n(\MO(n),\Z/2)$ the Thom class of $E(n)$, the Thom isomorphism theorem \cite[Theorem 10.2]{MS} shows that cup-product by $s_n$ induces an isomorphism $H^*(\BO(n),\Z/2)\isoto \wH^{*+n}(\MO(n),\Z/2)$. 
By Thom's definition of Stiefel--Whitney classes (see \cite[p.\,91]{MS}), one has 
\begin{equation}
\label{Thomdef}
\Sq^i(s_n)=w_i\cdot s_n.
\end{equation}
In particular, one computes that
\begin{equation}
\label{s2}
s_n^2=\Sq^n(s_n)=w_n\cdot s_n.
\end{equation}
Moreover, the Wu formula \cite[Problem 8-B]{MS} shows that 
\begin{equation}
\label{Wuformula}
\Sq^i(w_j)=\sum_{t=0}^i\binom{j+t-i-1}{t}w_{i-t}w_{j+t},
\end{equation}
where it is understood that $w_j=0$ in $H^*(\BO(n),\Z/2)$ for $j>n$.
These pieces of information and Cartan's formula entirely determine $H^{*}(\MO(n),\Z/2)$ as an algebra endowed with an action of the mod $2$ Steenrod algebra.

 Let $M$ be a compact $\ci$ manifold and let $x\in H^n(M,\Z/2)$ be a cohomology class. It follows from the Thom--Pontrjagin construction \cite[Th\'eor\`eme II.1]{Thom} that $x$ is Poincar\'e-dual to the fundamental class of a $\ci$ submanifold of codimension $n$ in $M$ if and only if there exists a continuous map $f:M\to \MO(n)$ such that $x=f^*s_n$. 

\vspace{.5em}

We also include a proof of the following well-known lemma for lack of a suitable reference.

\begin{lem}
\label{approxCW}
Let $K$ be a finite $CW$-complex.  For all $N\geq 0$, there exists a compact~$\ci$ manifold $M$ and an $N$-equivalence $f:M\to K$.
\end{lem}

\begin{proof}
By \cite[Theorem 2C.5]{Hatcher}, we may assume that $K$ is a finite simplicial complex. Viewing $K$ as a subcomplex of a simplex of high dimension,  we may embed it piecewise linearly in
$\R^d$ for $d\gg 0$.  Let $\Theta\subset\R^d$ be a neighbourhood of~$K$ that is a smooth regular neighbourhood as in \cite[Theorem 1 (a) and (a')]{Hirsch}, 
and set $M:=\partial\Theta$.  It follows from \cite[Corollary 1 p.\,725]{HZ} 
that the inclusion $M\hookrightarrow\Theta\setminus K$ is a weak homotopy equivalence,  the inclusion $\Theta\setminus K\hookrightarrow \Theta$ is an $N$\nobreakdash-equivalence because $K$ has high codimension in $\Theta$ when $d\gg0$,  
and the inclusion $K\hookrightarrow \Theta$ is a homotopy equivalence by choice of $\Theta$. These facts prove the lemma.
\end{proof}

We now reach the goal of this paragraph.

\begin{thm}
\label{cexSWsub}
There exists a compact $\ci$ manifold $M$ and an algebraizable class $x\in H^5(M,\Z/2)$ that does not belong to the subring $A(M)$ of $H^*(M,\Z/2)$ generated by Stiefel--Whitney classes of vector bundles on $M$ and Poincar\'e-duals of fundamental classes of $\ci$ submanifolds of $M$.
\end{thm}

\begin{proof}
Let $K$ be the Thom space of the tautological rank $n$ vector bundle on $\Grass(3,\R^{3+m})$.
Comparing \cite[Theorem 7.1 and Problem~7-B]{MS} and applying the Thom isomorphim theorem \cite[Theorem 10.2]{MS} shows that the restriction map $H^*(\MO(3),\Z/2)\to H^*(K,\Z/2)$ is an isomorphism in degree $\leq 18$ if $m\gg0$. As $K$ is a finite CW-complex, it follows from Lemma \ref{approxCW}
that there exists a compact~$\ci$ manifold~$M$ and a continuous map $f:M\to\MO(3)$ such that the pull-back map $f^*:H^*(\MO(3),\Z/2)\to H^*(M,\Z/2)$ is an isomorphism in degree $\leq 18$.

By the Thom--Pontrjagin construction, the class $f^*s_3$ is Poincar\'e-dual to the fundamental class of a $\ci$ submanifold of~$M$, hence is algebraizable by \cite[Corollary 4.5]{BT} or \cite[Theorem 1.1]{AKrelNash}. 
It follows Corollary \ref{BHSteencor} that the class $x:=f^*(w_2\cdot s_3)\in H^5(M,\Z/2)$, 
which is equal to $\Sq^2(f^*s_3)$ by (\ref{Thomdef}), is algebraizable. 

Assume for contradiction that $x\in A(M)$.
As $H^q(M,\Z/2)=0$ for $q\in\{1,2\}$
and $H^5(M,\Z/2)$ is generated by $x=f^*(w_2\cdot s_3)$ and by $y:=f^*((w_1^2+w_2)\cdot s_3)$, we deduce that there exists $z\in\{x,y\}$ that is either the fifth Stiefel--Whitney class of a vector bundle $E$ on $M$, or Poincar\'e-dual to the fundamental class of a $\ci$ submanifold of codimension $5$ in $M$.
In the first case, one has
$$\Sq^1(z)=\Sq^1(w_5(E))=w_1(E)\cdot w_5(E)=0$$
in $H^{6}(M,\Z/2)$,
where the second equality follows from (\ref{Wuformula}) and the third equality holds since $H^1(M,\Z/2)=0$. One then computes (using the Cartan formula, (\ref{Thomdef}) and (\ref{Wuformula})) that $\Sq^1(x)=f^*(w_3\cdot s_3)$ and $\Sq^1(y)=f^*((w_1^3+w_3)\cdot s_3)$. This is a contradiction since both classes are non-zero by the computation of $H^{*}(\MO(3),\Z/2)$ recalled above and by the injectivity of $f^*$ in degree $6$.

In the second case, the Thom--Pontrjagin construction shows the existence of a continuous map $g:M\to\MO(5)$ with $z=g^*s_5$. Using the Cartan formula, (\ref{Thomdef}) and (\ref{Wuformula}), one computes that
$$\Sq^2\Sq^1(s_5)\cdot s_5^2=\Sq^2(w_1\cdot s_5)\cdot s_5^2=w_1^3\cdot s_5^3+w_1w_2\cdot s_5^3=\Sq^1(s_5)^3+\Sq^1(s_5)\cdot\Sq^2(s_5)\cdot s_5.$$
Pulling back this equality by the map $g$, we deduce that 
\begin{equation}
\label{identityinMO5}
\Sq^2\Sq^1(z)\cdot z^2+\Sq^1(z)^3+\Sq^1(z)\cdot \Sq^2(z)\cdot z=0
\end{equation}
in $H^{18}(M,\Z/2)$.
Using the Cartan formula, (\ref{Thomdef}), (\ref{s2}) and (\ref{Wuformula}), one shows that 
\begin{equation*}
\begin{alignedat}{5}
\Sq^2\Sq^1(x)&\cdot x^2+\Sq^1(x)^3+\Sq^1(x)\cdot \Sq^2(x)\cdot x\\
&=f^*(w_1^2w_3s_3\cdot(w_2s_3)^2+(w_3s_3)^3+w_3s_3\cdot(w_1^2w_2s_3+w_1w_3s_3)\cdot w_2s_3)\\
&=f^*((w_1w_2w_3^4+w_3^5)\cdot s_3),
\end{alignedat}
\end{equation*}
and a similar lengthier computation shows that
\begin{equation*}
\begin{alignedat}{5}
\Sq^2\Sq^1(y)\cdot y^2&+\Sq^1(y)^3+\Sq^1(y)\cdot \Sq^2(y)\cdot y\\
&=f^*((w_1^3w_2^3w_3^2+w_1^2w_2^2w_3^3+w_1w_2w_3^4+w_3^5)\cdot s_3).
\end{alignedat}
\end{equation*}
As both these classes are non-zero by the computation of $H^{*}(\MO(3),\Z/2)$ recalled above and by the injectivity of $f^*$ in degree $18$, this contradicts (\ref{identityinMO5}).
\end{proof}

 \bibliographystyle{myamsalpha}
\bibliography{Steenrod}

\end{document}